\documentclass[11pt]{article} 
\usepackage[margin=1in]{geometry}
\usepackage[final]{showkeys}
\usepackage{subfigure}

\usepackage{authblk}

\usepackage{endnotes}

\usepackage{longtable}
\usepackage{tabularx}
\usepackage{booktabs}
\usepackage{etex}
\usepackage{multirow}
\usepackage{array}
\usepackage{bbm}

\usepackage[title]{appendix}

\usepackage{amsthm}

\newtheorem{theorem}{Theorem}
\newtheorem{lemma}{Lemma}

\newtheorem{proposition}{Proposition}
\newtheorem{definition}{Definition}

\newcommand{\vct}{\boldsymbol }
\newcommand{\mat}{\mathbf }

\newcommand{\op}{\mathrm{op}}

\renewcommand{\bar}{\overline}

\usepackage{amsbsy, amsfonts, amsgen, amsmath, amsopn, amssymb, amstext,
amsxtra, bezier, color, enumerate, graphicx, latexsym, verbatim,
pictexwd, supertabular, url, dsfont,leftidx, mathrsfs, appendix, setspace}
\usepackage{algorithm}
\usepackage{algorithmicx}
\usepackage{algpseudocode}
\makeatletter
\def\BState{\State\hskip-\ALG@thistlm}
\makeatother

\usepackage{natbib}
 \bibpunct[, ]{(}{)}{,}{a}{}{,}%
\usepackage[colorlinks=true,bookmarks=false,urlcolor=blue, citecolor=blue,linkcolor=blue,bookmarksopen=false,draft=false]{hyperref}

\newcommand{\R}{ \mathbb{R} }

\newcommand{\E}{ \mathbb{E} }

\usepackage{parskip} 

\newcommand{\BFA}{\mathbf A}
\newcommand{\BFB}{\mathbf B}
\newcommand{\BFC}{\mathbf C}
\newcommand{\BFD}{\mathbf D}
\newcommand{\BFZ}{\mathbf Z}
\newcommand{\BFr}{\mathbf r}
\newcommand{\BFd}{\mathbf d}
\newcommand{\BFf}{\mathbf f}
\newcommand{\BFp}{\mathbf p}
\newcommand{\BFx}{\mathbf x}
\newcommand{\BFy}{\mathbf y}
\newcommand{\BFz}{\mathbf z}
\newcommand{\BFu}{\mathbf u}
\newcommand{\BFxi}{\mathbf \xi}
\newcommand{\BFzero}{\mathbf 0}
\newcommand{\BFone}{\mathbf 1}
\newcommand{\BFlambda}{\mathbf \lambda}
\newcommand{\BFDelta}{\mathbf \Delta}

\newcommand{\Fluid}{\mathrm{F}}

\begin{document}



\title{Constant Regret Re-solving Heuristics for Price-based Revenue Management}
\author[1]{Yining Wang}
\author[2]{He Wang}
\affil[1]{Warrington College of Business, University of Florida}
\affil[2]{School of Industrial and Systems Engineering, Georgia Institute of Technology}

\maketitle

\begin{abstract}
Price-based revenue management is an important problem in operations management with many practical applications.
The problem considers a retailer who sells a product (or multiple products) over $T$ consecutive time periods and is subject to constraints
on the initial inventory levels.
While the optimal pricing policy could be obtained via dynamic programming,
such an approach is sometimes undesirable because of high computational costs.
Approximate policies, such as the re-solving heuristics, are often applied as computationally tractable alternatives.
In this paper, we show the following two results.
First, we prove that a natural re-solving heuristic attains $O(1)$ regret compared to
the value of the optimal policy.
This improves the $O(\ln T)$ regret upper bound established in the prior work of \cite{jasin2014reoptimization}.
Second, we prove that there is an $\Omega(\ln T)$ gap between the value of the optimal policy and that of the fluid model.
This complements our upper bound result by showing that the fluid is not an adequate information-relaxed benchmark
when analyzing price-based revenue management algorithms.

\end{abstract}


\section{Introduction}
We study a classic price-based revenue management problem where a retailer sells either a single product { or a group  of products} over a finite horizon given fixed initial inventory \citep{gallego1994optimal,gallego1997multiproduct}.
More specifically, consider $n$ products and $T$ consecutive selling periods 
with an initial inventory level $\BFy_0 \in \mathbb{R}^n_+$.
{
At time $t$, the retailer posts a price vector $\BFp_t\in \mathcal{P} \subset \mathbb{R}^n_+$. 
Suppose $\BFf: \mathcal{P} \to \mathbb{R}^n_{+}$ is a fixed demand function.
Let $(\Omega, \mathcal{F}, \{\mathcal{F}_t\}_{t=1}^{T}, \Pr)$ be a filtered probability space.}
The realized demand $\BFd_t$, realized revenue $\BFr_t$, 
 and remaining inventory level $\BFy_t$ at the end of period $t$ are governed by
 {
 \begin{equation}
 \BFd_t = \BFf(\BFp_t) + \BFxi_t, \quad r_t = \BFp_t^\top \min\{\BFd_t, \BFy_{t-1}\}, \quad \BFy_t = \max\{0, \BFy_{t-1}-\BFd_t\},
 \label{eq:model}
 \end{equation}
 where $\{\BFxi_t\}_{t=1}^T$ is a martingale difference sequence adapted to the filtration $\{\mathcal{F}_t\}_{t=1}^{T}$
 (see Sec.~\ref{sec:main-results} for detailed assumptions.)}
 
 The retailer's objective is to design an admissible pricing policy $\pi$ to maximize the expected revenue
 over the $T$ periods.
 A pricing policy $\pi$ can be represented by $\pi=(\pi_1,\cdots,\pi_T)$, where 
 $\pi_t$ is a mapping from the inventory level $\BFy_{t-1}$ to the price $\BFp_t\in \mathcal{P}$.
 {
 A pricing policy $\pi$ is \emph{admissible} or \emph{non-anticipating} if the posted price $\BFp_t$ only depends
 on the history up to the end of period $t-1$, namely, 
 $\BFp_t$ is measurable with respect to $\mathcal{F}_{t-1}$.}
 Given the initial inventory level $\BFx_T := \BFy_0 / T$,
 the expected revenue of an admissible policy $\pi$ is denoted by
 \begin{equation}
 R^\pi(T,\BFx_T) := \E\left[\sum_{t=1}^T r_t~\bigg|~ \BFp_t = \pi_t(\BFy_{t-1}),\ \forall t=1,\cdots,T\right].
 \label{eq:rpit}
 \end{equation}
 
 \subsection{Existing Results on the Fluid Model and the Re-solving Heuristic}
 
 An optimal policy $\pi^*$ maximizing $R^\pi(T,\BFx_T)$ defined in Eq.~(\ref{eq:rpit}) can be in principle obtained via dynamic programming (DP).
{ However, it is well known that the exact DP algorithm suffers from the curse of dimensionality, as the (discretized) state space grows exponentially in size with the number of products.}
 
The seminal work of \cite{gallego1994optimal} proposed a fluid approximation model of the optimal dynamic pricing problem.
To define the fluid model, 
suppose there exists an inverse function $\BFf^{-1}$ of the demand rate function of $\BFf$.
Let $r(\BFx) := \BFx^\top \BFf^{-1}(x)$ be the mean revenue in one period given the price vector $\BFf^{-1}(\BFx)$. 
{ 
We assume $r(\BFx)$ is strictly concave and smooth on its domain $\mathcal{D}\subset \mathbb{R}^n_+$ (see Sec.~\ref{sec:main-results} and Sec.~\ref{sec:multi-product} for the statements of these assumptions).
The fluid model for the dynamic pricing problem is
\begin{align}\label{eq:fluid_opt} 
    \max_{\BFx \in \mathcal{D}} \ &\ r(\BFx) \\
    \text{s.t.} \ &\ \BFzero \leq \BFx \leq \BFx_T. \nonumber
\end{align}
Let $r^\Fluid$ and $\BFx^\Fluid$  denote the optimal value and the (unique) optimal solution of the fluid model, respectively.
The optimization problem \eqref{eq:fluid_opt} chooses a demand rate $\BFx^\Fluid$ (or equivalently, setting the price to $\BFf^{-1}(\BFx^\Fluid)$) such that the revenue function is maximized subject to the initial inventory constraint.
Intuitively, the fluid approximation model ignores the randomness caused by demand noises and replaces the stochastic inventory constraint with a deterministic constraint.

}
The following results are established by \cite{gallego1994optimal,gallego1997multiproduct}.
\begin{theorem}[\cite{gallego1994optimal,gallego1997multiproduct}]
For any admissible policy $\pi$ and the initial inventory level $\BFy_0=\BFx_T T$, the expected revenue is upper bounded by $R^\pi(T,\BFx_T)\leq T r^\Fluid$.
Furthermore, for a static pricing policy $\pi^s$ such that $\BFp_t = \BFf^{-1}(\BFx^\Fluid)\ (\forall t=1,\cdots,T)$, the expected revenue is lower bounded by $R^{\pi^s}(T,\BFx_T) \geq Tr^\Fluid - O(\sqrt T)$.
\label{thm:reopt-root}
\end{theorem}

However, the main drawback of the static pricing policy is that it is not adaptive to demand randomness.
Researchers have proposed various approaches to modify the static pricing policy, aiming to
improve the $O(\sqrt{T})$ gap in Theorem \ref{thm:reopt-root} (see Sec.~\ref{sec:related_work} for a survey).
One intuitive approach is to \emph{re-solve} the fluid model in every period using the current inventory level: 
\begin{align*}
    \max_{\BFx \in \mathcal{D}} \ &\ r(\BFx) \\
    \text{s.t.} \ &\ \BFzero \leq \BFx \leq \BFx_t,
\end{align*}
where $\BFx_t := \BFy_{t-1}/(T-t+1)$ is the normalized inventory level realized at the beginning of period $t$. Let the solution to the above problem be $\BFx^c_t$ (the superscript $c$ stands for ``constrained"). The price is then reset to $\BFp_t= \BFf^{-1}(\BFx^c_t)$ at the start of period $t$.
We will refer to this policy as the \emph{re-solving heuristic}.
In the work of \cite{jasin2014reoptimization}, the gap is reduced from $O(\sqrt{T})$ to $O(\ln T)$ by using the re-solving heuristic, as shown by the following result.

\begin{theorem}[\cite{jasin2014reoptimization}]
Let $\pi^r=(\pi_1^r,\cdots,\pi_T^r)$ be the re-solving heuristic policy defined as $\BFp_t= \BFf^{-1}(\BFx^c_t)$. { Assume the optimal dual variables of the fluid model \eqref{eq:fluid_opt} are strictly positive.}
The expected revenue is bounded by $R^{\pi^r}(T,\BFx_T) \geq Tr^\Fluid-O(\ln T)$, where $\BFx_T = \BFy_0 / T$.
\label{thm:reopt-log}
\end{theorem}


\subsection{Our Results: Constant Regret and Logarithmic Gaps}

\begin{figure}[t]
\centering
\includegraphics[width=0.9\textwidth]{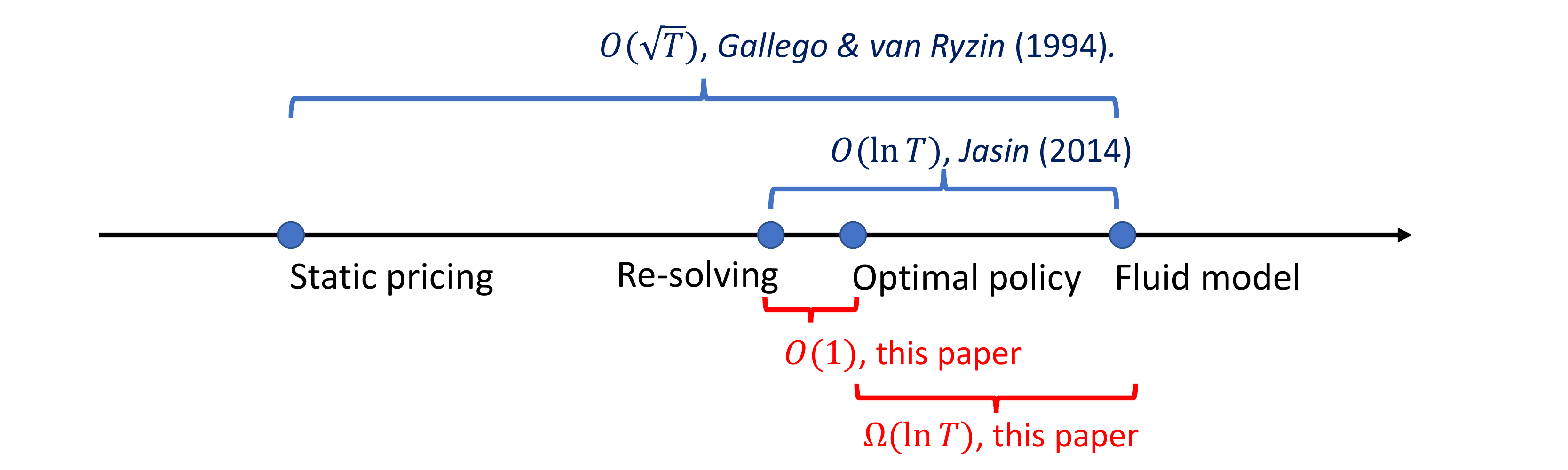}
\caption{An illustration of existing results (blue) compared to our results (red).}
\label{fig:results}
\end{figure}

In this paper we establish two main results: constant regret for the re-solving heuristic, and a logarithmic regret lower bound
on the gap between the value of the optimal pricing policy and the fluid model.
Figure \ref{fig:results} summarizes results established in this paper and compares them with existing results in the prior literature.

Our first main result, as stated in Theorems \ref{thm:main-ub} and \ref{thm:upper-bound-multi-product} later in the paper,
asserts that for any initial inventory level (except for certain boundary cases), the cumulative regret of the re-solving heuristic $\pi^r$
is upper bounded by a constant that is independent of $T$, 
compared against the expected reward of the optimal dynamic pricing policy $\pi^*$.
Apart from the obvious improvement from $O(\ln T)$ to $O(1)$ in regret bound,
our proof technique is quite different from existing work that compares the expected reward of the re-solving policy
to a certain information relaxed benchmark, such as the fluid model or the hindsight optimum benchmark.
Instead, 
we compare the value of $\pi^r$ directly with the value of the optimal DP policy $\pi^*$ by carefully analyzing the stochastic inventory levels under the two policies.

Our second main result, as stated in Theorem \ref{thm:main-lb} later in this paper, 
shows that there is an $\Omega(\ln T)$ \emph{lower bound} on the gap between the expected revenue of the re-solving heuristics $\pi^r$
and the optimal objective of the fluid model $Tr^\Fluid$.
Coupled with the $O(1)$ regret upper bound established in Theorem \ref{thm:main-ub},
this shows that there is an $\Omega(\ln T)$ lower bound on the gap between the value of the optimal policy $\pi^*$
and the fluid model as well.
This demonstrates a fundamental limitation of using the fluid model as a benchmark to analyze the regret for the price-based revenue management problem.

\section{Related Work}
\label{sec:related_work}

The idea of using simple, easy-to-compute pricing policies to approximate optimal dynamic pricing policies
originates from  \cite{gallego1994optimal,gallego1997multiproduct}, who proposed a static price policy using fluid models and established $O(\sqrt{T})$ regret.
{ Later, \cite{maglaras2006dynamic} showed that the re-solving heuristic has $o(T)$ regret in the single resource case.
}
The most relevant prior research to our paper is the work by \cite{jasin2014reoptimization},
who studied a price-based network revenue management problem and showed that re-solving heuristics attain $O(\ln T)$
asymptotic regret upper bound under mild conditions.
\cite{jasin2014reoptimization} also showed that infrequent re-solving has similar theoretical performance guarantees and is much more computationally efficient.
In this paper, we improve the regret of the re-solving heuristic to an $O(1)$ constant bound that is independent of $T$.
Our analysis is different from the one in \cite{jasin2014reoptimization} in the sense that we directly compare the expected revenue of re-solving
with the value of the optimal DP policy, instead of the fluid model.
Additionally, we complement the result in \cite{jasin2014reoptimization} by establishing an $\Omega(\ln T)$ lower bound between the expected revenue of the optimal DP policy
and the fluid model.

Re-solving heuristics have also been studied in quantity-based revenue management 
\citep{cooper2002asymptotic,reiman2008asymptotically,secomandi2008analysis,jasin2013analysis,bumpensanti2020re}.
{
The decisions involved in a quantity-based revenue management problem are the opening and closing of available products, so the feasible decisions form a discrete set.
In contrast, the price-based revenue management model studied in this paper assumes continuous prices in an infinite set. 
When the decision set is discrete, re-solving the fluid model will lead to fractional solutions that require rounding; it is shown that the design of  the rounding procedures (e.g.\ by randomization, thresholding, etc.) plays a critical role in the performance of re-solving heuristics \citep{arlotto2019uniformly,bumpensanti2020re,vera2020bayesian}.
When the decision set is continuous, re-solving algorithms must precisely track the updated inventory level in the fluid model without using any rounding.
As such, analysis for the (continuous) price-based revenue management problem is much different from the previous proofs for quantity-based revenue management.
Moreover, existing proofs of constant regret in the quantity-based model often use an information relaxation bound called \emph{hindsight-optimum} benchmark (see Appendix for a detailed discussion). However, the hindsight-optimum benchmark does not readily translate to the price-based revenue management setting, and our analysis does not rely on such a benchmark.
}
It is worth mentioning that \cite{vera2019online} proposed a constant regret algorithm that works for both quantity-based and price-based revenue management problems.
{
However, the price-based revenue management model considered in
\cite{vera2019online} assumes a single product and a finite set of candidate prices.
Because a finite set of prices implies that the feasible decision set has a discrete structure, the pricing setting studied in \cite{vera2019online} is significantly different from this paper.
In fact, \cite{maglaras2006dynamic} showed that a single-resource revenue management problem with discrete prices is equivalent to a quantity-based revenue management problem.
}

In this paper, we restrict our attention to stationary demand. Several other papers studied revenue management problems with non-stationary, time-correlated, or arbitrary demand sequences \citep[e.g.][]{chen2013simple,ma2020approximation,ma2020algorithms}. Revenue management problems with non-stationary demand are harder and therefore these papers consider different performance metrics such as approximation ratios or competitive ratios rather than regret.
\cite{chen2013simple} studied a dynamic pricing problem under a general non-stationary demand setting. They showed that re-solving heuristics can achieve constant competitive ratios, whereas  static pricing policies have asymptotically diminishing competitive ratios.
Due to different model assumptions, the results in \cite{chen2013simple} are not directly comparable to ours.

Another stream of related literature studies \emph{dynamic pricing with demand learning},
where the underlying demand function is unknown and needs to be learned on the fly from sales data.
{ Many papers consider demand learning settings using the price-based finite-inventory revenue management model from \cite{gallego1994optimal} as the ground truth model}
\citep[e.g.,][]{aviv2005dynamic,besbes2009dynamic,besbes2012blind,wang2014close,lei2014near,den2015dynamic,ferreira2018online}. { Re-solving heuristics have also been applied to demand learning algorithms \citep{jasin2015performance,ferreira2018online}.}
In contrast, our paper assumes that the retailer has {full information} about the demand curve and demand distributions, so we are able to show $O(1)$ regret, which is tighter than typical regret bounds in the learning literature. Also,
the lower bound result in this paper (Theorem~\ref{thm:main-lb}) is proved using different techniques
from lower bound proofs in the learning setting \citep[e.g.][]{broder2012dynamic,wang2019multi}, as the latter relies on information-theoretical lower bounds.

\section{Main Results}
\label{sec:main-results}

{ In this section, we consider a price-based revenue management problem with a single product. The extension to multiple products is deferred to Sec.~\ref{sec:multi-product}. Although some of the results in this section are special cases of those in Sec.~\ref{sec:multi-product}, presenting the single-product model helps illustrate the key insight of our analysis.}
We make the following assumptions on the single-product demand model.
\begin{enumerate}
\item[A1.] {(\bf Monotonicity)} The demand rate function $f:[\underline{p}, \overline{p}]\to \mathbb{R}_+$ is strictly decreasing with $f(\underline{p})=\overline{d}$, $f(\overline{p})=\underline{d}$. (We allow the case $\overline{p}=+\infty$ or $\underline{d}=0$.)
Note that this implies the existence of the inverse function $f^{-1}$ on $[\underline{d},\overline{d}]$.
\item[A2.] {(\bf Strict Concavity)} The expected revenue $r(d)=df^{-1}(d)$ as a function of the demand rate $d$ is strictly concave. There exists a positive constant $m>0$ such that $r''(d)\leq -m$ for all $d\in[\underline d,\overline d]$.
The maximizer of $r(d)$ is in the interior of the domain, i.e., $\arg\max_{d} r(d) \in (\underline d,\overline d)$.
\item[A3.] {(\bf Smoothness)} The third derivative of $r(d)$ exists and satisfies $|r'''(d)|\leq M$ for all $d\in [\underline d,\overline d]$.
Furthermore, there exists a constant $C>0$ such that $|r(d)-r({d}')|\leq C|d-{d}'|$ for all $d,{d}'\in[\underline d,\overline d]$.
{
\item[A4.] {(\bf Martingale Difference Sequence)}
Conditional on the price $p_t$ ($\forall t = 1,\ldots,T$), the demand noise $\xi_t$ is independent of $\{\xi_{1},\ldots,\xi_{t-1}\}$ and satisfies 
$\E[\xi_t \mid \mathcal{F}_{t-1}] = \E[\xi_t \mid p_t] =0 \ a.s.$ 
The conditional distribution of $\xi_t$ given $p_t$ is denoted by $\xi_t\sim Q(p_t)$. In addition, $|\xi_t|\leq B_\xi \ a.s.$ for some constant $B_\xi<\infty$.
\item[A5.] {\bf (Wasserstein Distance)} There exists a constant $L>0$ such that for any $p,p'\in[\underline{p}, \overline{p}]$, it holds that $\mathcal W_2(Q(p),Q(p'))\leq L|f(p)-f(p')|$,
where $\mathcal W_2(Q,Q'):= \inf_{\Xi}\sqrt{\E_{\xi,\xi'\sim\Xi}[|\xi-\xi'|^2]}$ is the $L_2$-Wasserstein distance
between $Q,Q'$, with $\Xi$ being an arbitrary joint distribution with marginal distributions being $Q$ and $Q'$, respectively.
}
\end{enumerate}
{ 
Assumptions (A1)--(A3) are standard assumptions for the price-based revenue management problem \citep{gallego1994optimal,jasin2014reoptimization}. In particular, the strict concavity of $r(d)$ stems from the economic principle of diminishing marginal returns.
Assumption (A4) allows demand noise to have general dependence on the price.
Assumption (A5) states that if two prices are close to each other, then the demand noises given these prices should also have similar distributions. This assumption is satisfied when the demand distribution is modeled by a parametric family of distributions (e.g., Bernoulli, binomial, truncated normal) whose parameters depend continuously on price.}

\subsection{Constant Regret of the Re-solving Heuristic}


Let $r^{u} := \max_{x\geq 0}r(x)$
and 
$x^{u} := \arg\max_{x\geq 0}r(x)$ be the unconstrained optimal revenue rate and its maximizer, respectively. 
Because $r(x)$ is strictly concave, in the case of a single product, it is easily verified that the unique optimal solution to the fluid model is equal to $x^\Fluid = \min\{x_T, x^{u}\}$ and hence the fluid optimal price is $f^{-1}(\min\{x_T, x^{u}\})$.

When the normalized initial inventory level $x_T$ exceeds the (unconstrained) optimal demand rate $x^{u}$,
it is easy to verify that both the static policy $\pi^s: p_t\equiv f^{-1}(x^{u})$ and the re-solving heuristic $\pi^r$ have constant regret.
\begin{proposition}
Suppose $x_T > x^{u}$. Let $\pi^s: p_t\equiv f^{-1}(x^{u})$ be the static pricing policy.
Then $R^{\pi^s}(T,x_T) \geq Tr(x^{u}) - O(1) \geq R^{\pi^*}(T, x_T) - O(1)$. { In addition, the re-solving heuristic satisfies $R^{\pi^r}(T,x_T) \geq Tr(x^{u}) - O(1) \geq R^{\pi^*}(T, x_T) - O(1)$.}
\label{prop:static-lp-const}
\end{proposition}

\begin{proof}{Proof of Proposition \ref{prop:static-lp-const}.}
{ 
For any $1\leq t \leq T$, 
let $x_t = y_{t-1}/t$ be the normalized inventory level in period $t$. The static pricing policy commits to the price $f^{-1}(x^{u})$.
If $x_t > x^{u}$, the re-solving heuristic also selects the price $f^{-1}(x^{u})$. 
Let $\pi$ denote either the static pricing policy $\pi^s$ or the re-solving heuristic $\pi^r$. 
Let $\mathcal{E}_t = \{\sup_{1\leq t' \leq t-1} \sum_{\tau=1}^{t'} \xi_\tau \leq T(x_T - x^{u})\}$ be the event that the inventory level never drops below $x_u$ from the start to the beginning of period $t$. 
Then, we have
\begin{align*}
R^{\pi}(T,x_T) & \geq \E\left[\sum_{t=1}^T (r(x^{u})+\xi_t)\BFone\{\mathcal{E}_t\}\right] = \E\left[\sum_{t=1}^T r(x^{u})\BFone\{\mathcal{E}_t\}\right]  = r(x^{u})\left(T -\sum_{t=1}^T \Pr[\mathcal E_t^c]\right),
\end{align*}
where the first equality uses the fact that $\{\xi_t\}$ is a martingale difference sequence and $\mathcal{E}_t \in \mathcal{F}_{t-1}$.
By Doob's martingale inequality, for any $1\leq t \leq T$, we have
\[
    \Pr[\mathcal E_t^c] = \Pr\left[\sup_{1\leq t' \leq t-1} \sum_{\tau=1}^{t'} \xi_\tau > T(x_T - x^{u})\right] \leq \frac{\mathbb \sum_{\tau=1}^t \E[\xi_\tau^2]}{T^2(x_T - x^{u})^2} \leq \frac{t B_{\xi}^2}{T^2(x_T - x^{u})^2}.
\]
Thus, $R^{\pi}(T,x_T) \geq Tr(x^{u}) - (B_{\xi}^2)/(x_T - x^{u})^2 = Tr(x^{u}) - O(1)$. The proof is complete by noting that $Tr(x^{u}) \geq R^{\pi^*}(T,x_T)$ using Theorem~\ref{thm:reopt-root}. $\square$
}
\end{proof}

However, analysis for the limited inventory case $(x_T < x^{u})$ is much more complicated.
The static price policy $\pi^s \equiv f^{-1}(x_T)$ typically suffers $\Omega(\sqrt{T})$ regret in this case.
The work by \cite{jasin2014reoptimization} established that the regret of the re-solving heuristic $\pi^r$ when measured against the fluid benchmark $Tr(x_T)$ is at most $O(\ln T)$. { (Note that the ``positive dual variable'' assumption in \cite{jasin2014reoptimization} (see Theorem~\ref{thm:reopt-log}) is equivalent to the condition $x_T < x^{u}$ in the single product case.)}
Our next theorem improves the regret of $\pi^r$ to a constant.
{
\begin{theorem}
Suppose $x_T\in (\underline d, x^{u})$. Let $\pi^r$ be the re-solving heuristic and $\pi^*$ be the optimal policy. For $T\geq2$,
it holds that $R^{\pi^r}(T,x_T) \geq R^{\pi^*}(T, x_T) - O(1)$.
\label{thm:main-ub}
\end{theorem} 
}
Theorem \ref{thm:main-ub} is the main result of this section and its proof is given in Sec.~\ref{sec:proofs}. 
Unlike the previous results by \cite{gallego1994optimal}
and \cite{jasin2014reoptimization}, Theorem \ref{thm:main-ub} compares the expected revenue of $\pi^r$
directly with the optimal DP pricing policy $\pi^*$,
rather than comparing it with the fluid approximation value $Tr(x_T)$. This allows for tighter regret bounds.
In fact, it is impossible to obtain $O(1)$ regret using the fluid model as a benchmark,
as we shall establish in the next subsection.

{ In light of Proposition~\ref{prop:static-lp-const} and Theorem~\ref{thm:main-ub}, we know that the re-solving heuristic $\pi^r$ has constant regret either in the sufficient inventory case ($x_T > x^u$) or in the limited inventory case ($x_T < x^u$).  The only remaining scenario is the boundary case ($x_T = x^u$). We will investigate this scenario using numerical experiments in Sec.~\ref{sec:numerical_results}. Our numerical result indicates that $\pi^r$ does \emph{not} have constant regret in the boundary case. This observation is analogous to the situation for the quantity-based revenue management problem, where the re-solving heuristic has constant regret when the fluid model (a linear program) has non-degenerate solutions but does not have constant regret in certain boundary cases when the fluid model has degenerate solutions  \citep{jasin2012re,bumpensanti2020re}.
}

\subsection{Logarithmic Gap of the Fluid Model Benchmark}

In this section, we show that in the limited inventory case ($x_T < x^{u}$),
the regret of the re-solving policy $\pi^r$ measured against the fluid approximation value $Tr(x_T)$
can be at least $\Omega(\ln T)$.
\begin{theorem}
Suppose $x_T\in(\underline d,x^{u})$ and let $\pi^r$ be the re-solving policy defined in Theorem \ref{thm:reopt-log}.
{ Suppose there exists $\sigma>0$ such that $\E_{\xi\sim Q(p)}[\xi^2]\geq \sigma^2$ for any $p\in[\underline{p}, \overline{p}]$.}
For sufficiently large $T$, it holds that $R^{\pi^r}(T, x_T) \leq Tr(x_T) - \Omega(\ln T)$.
\label{thm:main-lb}
\end{theorem}

Theorem \ref{thm:main-lb} is proved in Sec.~\ref{sec:proofs}.
It implies that the logarithmic regret by \cite{jasin2014reoptimization} (see Theorem~\ref{thm:reopt-log} above), which uses the fluid model benchmark, is tight and cannot be improved.
Because $R^{\pi^*}(T,x_T) - R^{\pi^r}(T,x_T) = O(1)$ by
Theorem~\ref{thm:main-ub}, 
Theorem \ref{thm:main-lb} also 
shows that there is an logarithmic gap $\Omega(\ln T)$ between the value of the optimal DP pricing policy
and the fluid approximation value. This is why our constant regret analysis does not use the fluid model as the benchmark.

Prior work on the quantity-based revenue management problem \citep{reiman2008asymptotically,
bumpensanti2020re,vera2019online} also considered different regret benchmarks that are tighter than the fluid model,
including various versions of the ``hindsight optimum'' benchmark. The ``hindsight optimum'' model assumes a clairvoyant who knows the \emph{aggregate} realized demands for the entire horizon at the start.
In the Appendix of this paper, we show that one version of the hindsight optimum benchmark proposed by 
\cite{vera2019online}
has $O(1)$ regret when measured against
the fluid approximation benchmark. By Theorem~\ref{thm:main-lb}, this hindsight optimum benchmark is also $\Omega(\ln T)$ away from the value of the optimal DP policy.

\section{Proofs}
\label{sec:proofs}

Before presenting our proof, we first define some notations.
{ For convenience, throughout this section, time indices are counted backwards: time index $t$ occurs when there are $t$ periods until the end of the selling horizon.}
Let $\phi_t^*(x)=R^{\pi^*}(t,x)$ and $\phi_t^r(x)=R^{\pi^r}(t,x)$ be the expected cumulative revenue of the 
optimal DP pricing policy $\pi^*$ and the re-solving policy $\pi^r$, respectively, when there are $t$ remaining periods and $xt$ units of remaining inventory.
For $1\leq \tau \leq T$, let $x_\tau^*$ and $x_\tau^r$ be the normalized inventory levels
under policy $\pi^*$ and $\pi^r$, when there are $\tau$ time periods remaining.
These notations are summarized in Table \ref{tab:notation}, with some additional notations being defined later in the proof.

\begin{table}[t]
\centering
\caption{Notations used in the proof.}
\scalebox{0.9}{
\begin{tabular}{l|l|l}
\hline
\bf Notation& \bf Definition& \bf Meaning\\
\hline
$x^{u}$& $x^{u}=\arg\max_{x\in[\underline d,\overline d]}r(x)$& the optimal demand rate without inventory constraints\\
$\phi_t^*(x)$& $\phi_t^*(x)=R^{\pi^*}(t, x)$ & reward of $\pi^*$ with $t$ periods and $xt$ inventory \\
$\phi_t^r(x)$& $\phi_t^r(x)=R^{\pi^r}(t, x)$ & reward of re-solving with $t$ periods and $xt$ inventory \\
$\mathcal F_t$& $\sigma$-algebra of $\{\xi_\tau\}_{\tau\geq t}$& all the events known at the end of period $t$\\
$x_\tau^*,x_\tau^r$& remaining inventory divide by $\tau$& normalized inventory levels under policy $\pi^*$ and $\pi^r$\\
$\Delta_\tau$& see Eq.~(\ref{eq:defn-dp})& the optimal demand correction with $\tau$ periods remaining\\
$\bar\Delta_{\to t}$& $\frac{\Delta_T}{T-1}+\frac{\Delta_{T-1}}{T-2}+\cdots+\frac{\Delta_{t+1}}{t}$& harmonic series of demand corrections up to $t$ \\
$\xi_\tau^*,\xi_\tau^r$& $\xi_\tau^* \sim Q(f^{-1}(x_\tau^* + \Delta_{\tau}))$, $\xi_\tau^r \sim Q(f^{-1}(x_\tau^r))$ & the stochastic demand noises at time $\tau$ under $\pi^*,\pi^r$  \\
$\bar\xi_{\to t}^*,\bar\xi_{\to t}^r$& $\frac{\xi_T}{T-1}+\frac{\xi_{T-1}}{T-2}+\cdots+\frac{\xi_{t+1}}{t}$&demand noises up to $t$,
under $\pi^*$ and $\pi^r$ \\
$\bar\xi_{\to t}^\delta$& $\bar\xi_{\to t}^r-\bar\xi_{\to t}^*$& difference in harmonic demand noise series\\
$\Xi_\tau$& joint distribution over $(\xi_\tau^*,\xi_\tau^r)$&
the joint distribution that minimizes $
\sqrt{\E_{\Xi_\tau}[|\xi_\tau^*-\xi_\tau^r|^2]}$\\
$T^\sharp$& see Eq.~(\ref{eq:defn-Tsharp})& stopping time such that $\{x_\tau^r\}_{\tau\leq T^\sharp}$ is well-behaved\\
\hline
\end{tabular}
}
\label{tab:notation}
\end{table}

The rest of this section is organized as follows. 
In the first subsection we establish some properties of the optimal policy $\pi^*$ and the re-solving policy $\pi^r$.
More specifically, we establish upper and lower bounds of the expected rewards $\phi_t^*(\cdot),\phi_t^r(\cdot)$
using the key quantities of $\{\bar\Delta_{\to t}\}$ (harmonic series of optimal demand corrections),
$\{\bar\xi_{\to t}^*,\bar\xi_{\to t}^r\}$ (harmonic series of stochastic noise variables) and $T^\sharp$ (a stopping time until which
the demand noise process is well-behaved).
We then proceed with the proofs of Theorems \ref{thm:main-ub}, \ref{thm:main-lb} by carefully analyzing the differences in the Taylor expansions of $\phi_t^*(\cdot),\phi_t^r(\cdot)$.

\subsection{Properties of the Optimal Policy and the
Re-solving Heuristic}

For any $\tau\geq 1$, let the random variable $x_\tau^*$ be the normalized inventory level (i.e., remaining inventory divided by remaining time) at period $\tau$ under policy $\pi^*$.
{
Let $y_\tau=x_\tau^*\tau$ and $y_{\tau-1}=x_{\tau-1}^*(\tau-1)$ be the actual inventory levels when $\tau$ and $(\tau-1)$ time periods are remaining. 
If the policy selects the demand rate $x^*_{\tau} + \Delta \in [\underline d,\overline d]$, we have $y_{\tau-1}=y_{\tau}-(x_\tau^*+\Delta+\xi_\tau)$, which implies $x_{\tau-1}^* = y_{\tau-1}/(\tau-1) = x_{\tau}^* - (\Delta+\xi_\tau)/(\tau-1)$.
Therefore, the value of the optimal policy $\pi^*$ is given by the following Bellman equation:
\[
\phi_\tau^*(x_\tau^*) = \max_{\Delta \in[\underline d - x_\tau^*,\overline d- x_\tau^*]} r(x_\tau^*+\Delta) + \E_{\xi_{\tau}\sim Q(f^{-1}(x_\tau^*+\Delta))} \left[\phi_{\tau-1}^*\bigg(x_\tau^* - \frac{\Delta+\xi_\tau}{\tau-1}\bigg) \right].
\]
Let $\Delta_{\tau}$ denote the maximizer of the above equation.
Let $\xi^*_{\tau}$ be the realized demand noise, which is drawn from the distribution
$Q(f^{-1}(x_\tau^* + \Delta_{\tau}))$ (by Assumption (A4)).}
The normalized inventory level in the next period is $x_{\tau-1}^* = x_\tau^* - (\Delta_\tau+\xi_\tau^*)/(\tau-1)$. 
The Bellman equation can be rewritten as
\begin{equation}
\phi_\tau^*(x_\tau^*) = r(x_\tau^* + \Delta_\tau) + \E\left[\phi_{\tau-1}^*\left(x_\tau^* - \frac{\Delta_\tau+\xi_\tau^*}{\tau-1}\right)\right]\qquad \forall\ 2\leq \tau \leq T,\ x^*_\tau \geq 0.
\label{eq:defn-dp}
\end{equation}
{ The above equation implies $\underline d - x_\tau^* \leq  \Delta_\tau \leq \overline d - x_\tau^*$, but we remark that it does not required $x_\tau^*$ to be in the domain $[\underline{d}, \overline{d}]$.}

For any $t<T$, let $\bar\Delta_{\to t} := \frac{\Delta_T}{T-1}+\frac{\Delta_{T-1}}{T-2}+\cdots+\frac{\Delta_{t+1}}{t}$
and $\bar\xi_{\to t}^* := \frac{\xi_T}{T-1}+\frac{\xi_{T-1}^*}{T-2}+\cdots+\frac{\xi_{t+1}^*}{t}$ be the harmonic series of demand noises up to time $t$
under the optimal policy.
For $t=T$, we also define $\bar\Delta_{\to t}=\bar\xi^*_{\to t}=0$.
It then holds that
\begin{equation}
x_\tau^* = x_T - \bar\Delta_{\to\tau} - \bar\xi_{\to\tau}^*, \qquad \forall\ 1\leq\tau\leq T.
\label{eq:defn-xtau-star}
\end{equation}

Next, we consider the re-solving heuristic $\pi^r$.
For any $\tau\geq 1$, let the random variable $x_\tau^r$ be the normalized inventory level under the re-solving heuristic.
If $x_\tau^r\in [\underline d, x^{u}]$, the re-solving heuristic selects the price $f^{-1}(x_\tau^r)$. { Let $\xi^r_\tau$ be the realized demand noise under this price, which is drawn from the distribution
$Q(f^{-1}(x_\tau^r))$.
The normalized inventory level in the next period is equal to $x_{\tau-1}^r = (x_{\tau}\tau - x_{\tau}-\xi^r_{\tau})/(\tau-1)=x^r_{\tau}
- \xi^r_{\tau}/(\tau-1)$.
}
Thus, the value of the re-solving policy $\pi^r$ can be written as
\begin{equation}
\phi_t^r(x_\tau^r) = r(x_\tau^r) + \E\bigg[\phi_{t-1}^r\bigg(x_\tau^r - \frac{\xi_\tau^r}{\tau-1}\bigg)\bigg]\qquad
\forall\ 2\leq \tau \leq T,\ \underline{d} \leq x^r_\tau \leq x^u.
\label{eq:defn-resolving}
\end{equation}
Note that Eq.~(\ref{eq:defn-resolving}) does \emph{not} hold for $x_\tau^r>x^{u}$, in which case the re-solving policy $\pi^r$
would commit to the unconstrained optimal demand rate $x^{u}$ instead of $x_\tau^r$.
{This motivates the definition of a certain stopping time $T^\sharp$ in Eq.~(\ref{eq:defn-Tsharp}) below, which ensures that Eq.~(\ref{eq:defn-resolving}) holds for all $\tau\geq T^\sharp$.
}
Comparing Eq.~(\ref{eq:defn-resolving}) with Eq.~(\ref{eq:defn-dp}), we remark that the re-solving heuristics $\pi^r$ can be viewed a special case
of the dynamic programming policy with the decision rule restricted to $\Delta_\tau\equiv 0$ for all $x_\tau^r\in [\underline d, x^{u}]$.

Define a time index $T^\sharp$ as
{
\begin{equation}
    T^\sharp := \max \left\{\tau \geq 1:\;\; \big|\bar\xi_{\to{\tau-1}}^r\big| > \min\left(x_T-\underline d,x^{u}-x_T,-\frac{r'(x_T)}{r''(x_T)}\right) \right\} \vee 2,
    \label{eq:defn-Tsharp}
\end{equation}
where $\bar\xi_{\to\tau}^r = \frac{\xi_T^r}{T-1}+\cdots+\frac{\xi_{\tau+1}^r}{\tau}$ is the harmonic series of demand noises
under the re-solving policy.
Note that $\bar\xi_{\to{\tau-1}}^r$ is a function of $\{\xi^r_t\}_{t=\tau}^T$ and measurable with respect to $\mathcal{F}_{\tau}$, so $T^{\sharp}$ is indeed a stopping time.
}
Because of the bound $|\bar\xi_{\to \tau}|\leq \min\{x_T-\underline d,x^{u}-x_T\}$ holds for all $\tau \geq T^\sharp -1$, we have
\begin{equation}
x_\tau^r = x_T - \bar\xi_{\to\tau}^r, \quad \forall \tau\geq T^\sharp - 1.
\label{eq:defn-xtau-r}
\end{equation}
Intuitively, $T^\sharp$ is the first time that the inventory level in the next period $x^r_{\tau-1}$ falls outside of the interval $[\underline{d},x^{u}]$.
In other words, the inventory levels of the re-solving heuristic $x^r_{\tau}$ satisfies Eq.~\eqref{eq:defn-resolving} up to time $T^\sharp$.
{ The additional term $-r'(x_T)/r''(x_T)$ in Eq.~\eqref{eq:defn-Tsharp} is needed for technical reasons in the proof (this term is equal to $x^u - x_T$ when $r(d)$ is a quadratic function).}


{
\begin{lemma} \label{lem:T-sharp-ub}
For $x_T\in(\underline{d},x^{u})$ and any $T\geq 1$, it holds that 
\begin{equation}
    \E[T^\sharp]\leq 2+\frac{4B_\xi^4}{\min\{x_T-\underline d,x^{u}-x_T,-r'(x_T)/r''(x_T)\}^4}.
\label{eq:Tsharp-upper-bound}
\end{equation}
\end{lemma}

\begin{proof}{Proof of Lemma \ref{lem:T-sharp-ub}.}
Let $\gamma:=\min\{x_T-\underline d,x^{u}-x_T,-r'(x_T)/r''(x_T)\}$.
By the definition from Eq.~\eqref{eq:defn-Tsharp}, for $\tau > 2$, we have
\[
    \Pr[T^\sharp \geq \tau] = \Pr\left[\sup_{\tau-1\leq t \leq T} |\bar{\xi}^r_{\to t}| > \gamma \right]
    \leq \frac{\E[|\bar{\xi}^r_{\to\tau-1}|^4]}{\gamma^4},
\]
which follows from Doob's martingale inequality by noting that $\{\xi_{\tau}/\tau\}$ is a martingale difference sequence (see Assumption (A4)).
Recall that $\bar\xi_{\to t}^r = \frac{\xi_T^r}{T-1}+\cdots+\frac{\xi_{t+1}^r}{t}$ with $|\xi^r_t|\leq B_{\xi}\ a.s.$
For $\tau >2$, it is easily verified that
\begin{align*}
\E[|\bar\xi_{\to {\tau - 1}}^r|^4]
= \sum_{\tau \leq j,k \leq T}\frac{\E[|\xi_j^r|^2|\xi_k^r|^2]}{(j-1)^2(k-1)^2}
\leq  \left(\sum_{\tau \leq j \leq T}\frac{B_\xi^2}{(j-1)^2}\right)\left(\sum_{\tau \leq k \leq T}\frac{B_\xi^2}{(k-1)^2}\right) \leq \frac{4B_\xi^4}{(\tau-1)^2}.
\end{align*}
Using the equality $\E[T^\sharp] =\sum_{\tau=1}^T\Pr[T^\sharp \geq \tau]$, we have
\[
    \E[T^\sharp]\leq \sum_{\tau=1}^T \frac{\E[|\bar{\xi}^r_{\to\tau-1}|^4]}{\gamma^4} \leq 2 + \sum_{\tau=3}^T \frac{4B_\xi^4}{\gamma^4(\tau-1)^2}
    \leq 2 + \frac{4B_{\xi}^4}{\gamma^4}. 
\]
$\square$
\end{proof}
}

\begin{lemma}
Let $x_T\in(\underline{d},x^{u})$.
Let $T^\sharp$ be the stopping time defined in Eq.~\eqref{eq:defn-Tsharp}.
Then, it holds that 
\begin{align}
\E\left[\sum_{\tau=T^\sharp}^T r(x_\tau^*+\Delta_\tau) \right] &\leq  \phi_T^*(x_T) \leq  \E\left[\sum_{\tau=T^\sharp}^T r(x_\tau^*+\Delta_\tau) + (T^\sharp-1)r(\min\{x_{T^\sharp-1}^*,x^{u}\}) 
\right], \nonumber\\
\E\left[\sum_{\tau=T^\sharp}^T r(x_\tau^r) \right]  &\leq  \phi_T^r(x_T)  \leq \E\left[\sum_{\tau=T^\sharp}^T r(x_\tau^r) + 
(T^\sharp-1)r(\min\{x_{T^\sharp-1}^r,x^{u}\})\right].
\label{eq:ft-r-expansion}
\end{align}
\label{lem:ft-star-expansion}
\end{lemma}
\begin{proof}{Proof of Lemma \ref{lem:ft-star-expansion}.}
For the optimal DP policy,
the revenue collected in periods $T,T-1,\cdots,T^\sharp$ is $\sum_{\tau=T^\sharp}^T[r(x_\tau^*+\Delta_\tau) + f^{-1}(x_\tau^*+\Delta_\tau)\xi_\tau^*]$.
At the beginning of period $T^\sharp-1$, the remaining inventory level is $x_{T^\sharp-1}^* (T^\sharp-1)$.
Therefore, 
\begin{align*}
\phi_T^*(x_T)
&= \E\left[\sum_{\tau=T^\sharp}^T[r(x_\tau^*+\Delta_\tau) + f^{-1}(x_\tau^*+\Delta_\tau)\xi_\tau^*] + \phi_{T^\sharp-1}^*(x_{T^\sharp-1}^*)\right]\\
&=   \E\left[\sum_{\tau=T^\sharp}^T r(x_\tau^*+\Delta_\tau) +  \phi_{T^\sharp-1}^*(x_{T^\sharp-1}^*) \right],
\end{align*}
where the second equality holds because 
$\E[\sum_{\tau=T^\sharp}^T f^{-1}(x_\tau^*+\Delta_\tau)\xi_\tau^*]=0$
by applying Doob's optional stopping theorem.
Similarly, for the re-solving policy, 
because the revenue collected in periods $T,T-1,\cdots,T^\sharp$ is $\sum_{\tau=T^\sharp}^T[r(x_\tau^r) + f^{-1}(x_\tau^r)\xi_\tau^r]$, we have
\begin{align*}
\phi_T^r(x_T^r) &= \E\left[\sum_{\tau=T^\sharp}^T [r(x_\tau^r)+f^{-1}(x_\tau^r)\xi_\tau^r] + \phi_{T^\sharp-1}^r(x_{T^\sharp-1}^r)\right]\\
&= \E\left[\sum_{\tau=T^\sharp}^T r(x_\tau^r) + \phi_{T^\sharp-1}^r(x_{T^\sharp-1}^r)\right]. 
\end{align*}
To complete the proof, note that for any admissible policy $\pi$, we have $0\leq \phi^{\pi}_{T^\sharp-1}(x_{T^\sharp-1})
\leq (T^\sharp-1)r(x^c_{T^\sharp-1})
=(T^\sharp-1)r(\min\{x_{T^\sharp-1},x^{u}\})$ by Theorem~\ref{thm:reopt-root}.
 $\square$
\end{proof}

\subsection{Proof of Theorem \ref{thm:main-ub}}

In this section we prove Theorem \ref{thm:main-ub}. 
{By Lemma~\ref{lem:ft-star-expansion},
for any $x_T\in(\underline d,x^{u})$, it holds that
\begin{align}
\phi_T^*(x_T) & \leq \textstyle  \E\left[\sum_{\tau=T^\sharp}^T r(x_\tau^*+\Delta_\tau) + (T^\sharp-1)r(\min\{x^*_{T^\sharp-1},x^{u}\})\right], \label{eq:exp-1}\\
\phi_T^r(x_T) & \geq \textstyle \E\left[\sum_{\tau=T^\sharp}^T r(x_\tau^r) \right].
\label{eq:exp-2}
\end{align}
}
Recall that 
$T^\sharp$ is the stopping time defined in Eq.~(\ref{eq:defn-Tsharp}).
For any $\tau \geq T^\sharp-1$, 
we have $x^*_\tau = x_T - \bar\Delta_{\to\tau} - \bar \xi^*_{\to\tau}$ by Eq.~\eqref{eq:defn-xtau-star}, $x^r_\tau = x_T - \bar\xi^r_{\to\tau}$ by Eq.~\eqref{eq:defn-xtau-r}.
For notational simplicity, define $\bar\xi_{\to\tau}^\delta := \bar\xi_{\to\tau}^r-\bar\xi_{\to\tau}^*$.
So,
{ \begin{align}
r(x_\tau^*+\Delta_\tau) - r(x_\tau^r)
& \leq r'(x_T-\bar\xi_{\to\tau}^r)[\Delta_\tau-\bar\Delta_{\to\tau}+\bar\xi_{\to\tau}^\delta] - \frac{m}{2}|\Delta_\tau-\bar\Delta_{\to\tau}+\bar\xi_{\to\tau}^\delta|^2\nonumber\\
&\leq r'(x_T)[\Delta_\tau-\bar\Delta_{\to\tau}+\bar\xi_{\to\tau}^\delta] - r''(x_T)\bar\xi_{\to\tau}^r[\Delta_\tau-\bar\Delta_{\to\tau}+\bar\xi_{\to\tau}^\delta] \nonumber\\
&\;\;\;\; + \frac{M}{2}|\bar\xi_{\to\tau}^r|^2|\Delta_\tau-\bar\Delta_{\to\tau}+\bar\xi_{\to\tau}^\delta| - \frac{m}{2}|\Delta_\tau-\bar\Delta_{\to\tau}+\bar\xi_{\to\tau}^\delta|^2, 
\label{eq:proof-main-ub-1}
\end{align}
where the first inequality uses the strict concavity condition (A2) and the second inequality uses the smoothness condition (A3) of $r(d)$.

Define $\xi_\tau^\delta := \xi_\tau^r-\xi_\tau^*$. Also, $\bar\xi_{\to\tau}^\delta$ is measurable with respect to $\mathcal{F}_{\tau+1}$ (note: $\mathcal{F}_{\tau+1} \subset \mathcal{F}_{\tau}$ as time are indexed backwards).
By Doob's optional stopping theorem, $\E[\sum_{\tau=T^\sharp}^T \xi_{\tau}^\delta] = 0$ 
and $\E[ \sum_{\tau=T^\sharp}^T \bar\xi_{\to\tau}^r\xi_\tau^\delta] = 0$.
Taking expectations on both sides of Eq.~(\ref{eq:proof-main-ub-1})
and summing over $\tau = T,T-1,\cdots,T^\sharp$, we have
\begin{align}
\E\bigg[\sum_{\tau=T^\sharp}^T \bigl( r(x_\tau^*+\Delta_\tau)-r(x_\tau^r)\bigr)\bigg]
\leq &\  \E\bigg[ \sum_{\tau=T^\sharp}^T
\bigg(r'(x_T)[\Delta_\tau-\bar\Delta_{\to\tau}+\bar\xi_{\to\tau}^\delta-\xi_{\tau}^\delta] - r''(x_T)\bar\xi_{\to\tau}^r[\Delta_\tau-\bar\Delta_{\to\tau}]\nonumber \nonumber \\
+r''(x_T)\bar\xi_{\to\tau}^r [\xi_{\tau}^\delta-\bar\xi_{\to\tau}^\delta]&
+ \frac{M}{2}|\bar\xi_{\to\tau}^r|^2|\Delta_\tau-\bar\Delta_{\to\tau}+\bar\xi_{\to\tau}^\delta| - \frac{m}{2}|\Delta_\tau-\bar\Delta_{\to\tau}+\bar\xi_{\to\tau}^\delta|^2
\bigg)\bigg].
\label{eq:proof-main-ub-2}
\end{align}
Define $\eta^{*} = -(x^*_{T^\sharp -1}-x^{u})^+$, $\eta^{r} = (r'(x_T)/r''(x_T) - \bar\xi^{r}_{\to T^\sharp -1})^+$ and
$\eta = \eta^{*} + \eta^{r}$.
By applying Eq.~\eqref{eq:proof-main-ub-1} to period $\tau = T^\sharp -1$, we have
\begin{align}
&\ r(\min\{x_{T^\sharp -1}^*,x^u\})
-r(x_{T^\sharp -1}^r-\eta^r) 
= r(x_{T^\sharp -1}^*+\eta^{*})
-r(x_{T^\sharp -1}^r-\eta^r)\nonumber\\
\leq &\  r'(x_T)[\eta-\bar\Delta_{\to T^\sharp -1}+\bar\xi_{\to T^\sharp -1}^\delta] - r''(x_T)(\bar\xi_{\to T^\sharp -1}^r+\eta^r)[\eta-\bar\Delta_{\to T^\sharp -1}+\bar\xi^{\delta}_{\to T^\sharp -1}] \nonumber\\
& + \frac{M}{2}|\bar\xi_{\to T^\sharp -1}^r+\eta^r|^2| \eta-\bar\Delta_{\to T^\sharp -1}+\bar\xi_{\to T^\sharp -1}^\delta| - \frac{m}{2}|\eta-\bar\Delta_{\to T^\sharp -1}+\bar\xi_{\to T^\sharp -1}^\delta|^2 \nonumber\\
\leq &\  r'(x_T)[\eta-\bar\Delta_{\to T^\sharp -1}+\bar\xi_{\to T^\sharp -1}^\delta] - r''(x_T)(\bar\xi_{\to T^\sharp -1}^r+\eta^r)[\eta-\bar\Delta_{\to T^\sharp -1}+\bar\xi^{\delta}_{\to T^\sharp -1}] + \frac{M^2}{8m}|\bar\xi_{\to T^\sharp-1}^r|^4\nonumber\\
& - \frac{M^2}{8m}|\bar\xi_{\to T^\sharp-1}^r|^4+ \frac{M}{2}|\bar\xi_{\to T^\sharp -1}^r|^2| \eta-\bar\Delta_{\to T^\sharp -1}+\bar\xi_{\to T^\sharp -1}^\delta| - \frac{m}{2}|\eta-\bar\Delta_{\to T^\sharp -1}+\bar\xi_{\to T^\sharp -1}^\delta|^2 \nonumber\\
\leq &\  r'(x_T)[\eta-\bar\Delta_{\to T^\sharp -1}+\bar\xi_{\to T^\sharp -1}^\delta] - r''(x_T)(\bar\xi_{\to T^\sharp -1}^r+\eta^r)[\eta-\bar\Delta_{\to T^\sharp -1}+\bar\xi^{\delta}_{\to T^\sharp -1}] + \frac{M^2}{8m}|\bar\xi_{\to T^\sharp-1}^r|^4,
\label{eq:proof-at-T-sharp}
\end{align}
where the second inequality holds because $|\bar\xi_{\to T^\sharp -1}^r+\eta^r|\leq |\bar\xi_{\to T^\sharp -1}^r|$ by definition and the last inequality follows by completing the square, because $- \frac{M^2}{8m}|\bar\xi_{\to T^\sharp-1}^r|^4+ \frac{M}{2}|\bar\xi_{\to T^\sharp -1}^r|^2| \eta-\bar\Delta_{\to T^\sharp -1}+\bar\xi_{\to T^\sharp -1}^\delta| - \frac{m}{2}|\eta-\bar\Delta_{\to T^\sharp -1}+\bar\xi_{\to T^\sharp -1}^\delta|^2 = -\frac{m}{2}(|\eta-\bar\Delta_{\to T^\sharp-1}+\bar\xi_{\to T^\sharp-1}^\delta| - \frac{M}{2m}|\bar\xi_{\to T^\sharp-1}^r|^2)^2 \leq 0$ almost surely.

Subtracting \eqref{eq:exp-2} from \eqref{eq:exp-1}, we obtain
\begin{align}
& \quad \phi_T^*(x_T) - \phi_T^r(x_T) \nonumber \\
& \leq
\E\left[
\sum_{\tau=T^\sharp}^T \bigl(r(x^*_\tau+\Delta_\tau) - r(x^r_{\tau})\bigr) + (T^\sharp-1)r(\min\{x^*_{T^\sharp-1},x^{u}\})\right] \nonumber \\
& \leq
\E\left[
\sum_{\tau=T^\sharp}^T \bigl(r(x^*_\tau+\Delta_\tau) - r(x^r_{\tau})\bigr) + (T^\sharp-1)\bigl(r(x^*_{T^\sharp-1}+\eta^*)-r(x^r_{T^\sharp-1}-\eta^r)\bigr)\right] + \E\left[(T^\sharp-1)r(x^u)\right] \nonumber \\
& \leq \E\big[r'(x_T)\mathcal A - r''(x_T)\mathcal B + r''(x_T)\mathcal C + \mathcal D \big] + \E\left[(T^\sharp-1)r(x^{u})\right],
\label{eq:proof-main-ub-3}
\end{align}
where the second inequality holds because $x_{T^\sharp-1}+\eta^*=\min\{x_{T^\sharp-1}^*,x^u\}$ by definition, and $x^u$ is the maximizer of $r(\cdot)$.
The last inequality uses \eqref{eq:proof-main-ub-2} \eqref{eq:proof-at-T-sharp}
and the terms $\mathcal A,\mathcal B,\mathcal C,\mathcal D$ are defined as
\begin{align*}
\mathcal A \ &=\ \textstyle\sum_{\tau=T^{\sharp}}^T [\Delta_\tau-\bar\Delta_{\to\tau}+\bar\xi_{\to\tau}^\delta-\xi_{\tau}^\delta] - (T^\sharp-1)(\bar\Delta_{\to T^\sharp-1} - \bar\xi_{\to T^\sharp -1}^\delta) + (T^\sharp-1)\eta, \\
\mathcal B \ &=\ \textstyle\sum_{\tau=T^{\sharp}}^T\bar\xi_{\to\tau}^r[\Delta_\tau-\bar\Delta_{\to\tau}]
- (T^\sharp-1) \bar\xi_{\to T^\sharp-1}^r \bar\Delta_{\to T^\sharp-1} + (T^\sharp-1)(\bar\xi_{\to T^\sharp-1}^r+\eta^r)\eta, \\
\mathcal C \ &=\ \textstyle\sum_{\tau=T^{\sharp}}^T\bar\xi_{\to\tau}^r[\xi_\tau^\delta - \bar\xi_{\to\tau}^\delta]
-(T^\sharp-1) (\bar\xi_{\to T^\sharp-1}^r+\eta^r) \bar\xi_{\to T^\sharp-1}^\delta, \\
\mathcal D \ &=\ \textstyle\sum_{\tau=T^{\sharp}}^T \left(\frac{M}{2}|\bar\xi_{\to\tau}^r|^2|\Delta_\tau-\bar\Delta_{\to\tau}+\bar\xi_{\to\tau}^\delta| - \frac{m}{2}|\Delta_\tau-\bar\Delta_{\to\tau}+\bar\xi_{\to\tau}^\delta|^2\right)  
 + (T^\sharp-1) \frac{M^2}{8m}|\bar\xi_{\to T^\sharp-1}^r|^4.
\end{align*}
}
Next, we analyze the four terms $\mathcal A,\mathcal B,\mathcal C,\mathcal D$ separately.
Recall the definition  $\bar\Delta_{\to\tau}=\frac{\Delta_{T}}{T-1}+\cdots+\frac{\Delta_{\tau+1}}{\tau}$, $\overline\xi_{\to\tau}^\delta = \frac{\xi_T^\delta}{T-1}+\cdots+\frac{\xi_{\tau+1}^\delta}{\tau}$.
For the term $\mathcal A$, with elementary algebra, it can be verified that
{
\begin{equation}
 \textstyle\sum_{\tau=T^{\sharp}}^T [\Delta_\tau-\bar\Delta_{\to\tau}+\bar\xi_{\to\tau}^\delta-\xi_{\tau}^\delta] - (T^\sharp-1)(\bar\Delta_{\to T^\sharp-1} - \bar\xi_{\to T^\sharp -1}^\delta) = 0,
\label{eq:A-final}
\end{equation}
which implies that $\E[\mathcal A]= \E[(T^\sharp-1)\eta].$
}

Re-organizing all terms in $\mathcal B$ by $\xi_{t}\ (\forall t\geq T^\sharp)$, we obtain
{
\begin{align*}
& \textstyle\sum_{\tau=T^{\sharp}}^T\bar\xi_{\to\tau}^r[\Delta_\tau-\bar\Delta_{\to\tau}] - (T^\sharp-1) \bar\xi_{\to T^\sharp-1}^r\bar\Delta_{\to T^\sharp-1}\\
= &\ \textstyle\sum_{t=T^\sharp}^T \frac{\xi_t^r}{t-1}\sum_{\tau=T^{\sharp}}^{t-1}(\Delta_\tau-\bar\Delta_{\to\tau}) - (T^\sharp-1) \bar\xi_{\to T^\sharp-1}^r\bar\Delta_{\to T^\sharp-1} \\
= &\ \textstyle \sum_{t=T^\sharp}^T\frac{\xi_t^r}{t-1}\left[-(t-1)\bar\Delta_{\to t-1}+(T^\sharp-1) \bar\Delta_{\to T^\sharp-1} \right] - (T^\sharp-1) \bar\xi_{\to T^\sharp-1}^r\bar\Delta_{\to T^\sharp-1} \\
=&\ \textstyle -\sum_{t=T^\sharp}^T \xi_t^r\bar\Delta_{\to t-1},
\end{align*}
where the last equality uses
$\overline\xi_{\to T^\sharp-1}^\delta = \frac{\xi_T^\delta}{T-1}+\cdots+\frac{\xi_{T^\sharp}^\delta}{T^\sharp-1}$.
}
Note that the random variable $\bar\Delta_{\to \tau-1}=\frac{\Delta_T}{T-1}+\cdots + \frac{\Delta_\tau}{\tau-1}$ is measurable with respect to $\mathcal{F}_{\tau+1}$, since the DP policy is non-anticipating.
By Doob's optional stopping theorem, we have
{
$\E[
-\sum_{\tau=T^\sharp}^T \xi_\tau^r\bar\Delta_{\to \tau-1}]=0$ and thus
\begin{equation}
\textstyle
\E[\mathcal B] = \E\left[
-\sum_{\tau=T^\sharp}^T \xi_\tau^r\bar\Delta_{\to \tau-1}\right]
+ \E\left[(T^\sharp-1)(\bar\xi_{\to T^\sharp-1}^r+\eta^r)\eta\right]
=  \E\left[(T^\sharp-1)(\bar\xi_{\to T^\sharp-1}^r+\eta^r)\eta\right].
\label{eq:B-final}
\end{equation}
}
{
Next we analyze the term $\mathcal C$. Note that $\mathcal C$ has a similar structure as $\mathcal B$; therefore 
\begin{align*}
    \E[\mathcal C] 
= &\  \E\left[-\sum_{t=T^\sharp}^T \xi_t^r\bar\xi_{\to t-1}^\delta\right] - \E\left[(T^\sharp-1)\eta^r\bar\xi^{\delta}_{\to T^\sharp-1}\right]
= \E\left[-\sum_{t=T^\sharp}^T \xi_t^r(\frac{\xi_T^\delta}{T-1}+\cdots+\frac{\xi_{t+1}^\delta}{t}+\frac{\xi_t^\delta}{t-1})\right] \\
& - \E\left[(T^\sharp-1)\eta^r\bar\xi^{\delta}_{\to T^\sharp-1}\right] 
= \E\left[-\sum_{t=T^\sharp}^T \frac{\xi_t^r \xi_t^\delta}{t-1} \right]
- \E\left[(T^\sharp-1)\eta^r\bar\xi^{\delta}_{\to T^\sharp-1}\right],
\end{align*}
where the last equality uses Doob's optional stopping theorem.
Because $|\xi_t^r|\leq B_{\xi}\ a.s.$, we have
\begin{align}
& \left|\E\left[-\sum_{t=T^\sharp}^T \frac{\xi_t^r \xi_t^\delta}{t-1} \right]\right|
\leq \E\left[\sum_{t=T^\sharp}^T \frac{|\xi_t^r| |\xi_t^r-\xi_t^*|}{t-1} \right]
\leq B_{\xi} \E\left[\sum_{t=T^\sharp}^T \frac{ |\xi_t^r-\xi_t^*|}{t-1} \right]
= B_{\xi} \E\left[\sum_{t=2}^T \vct{1}\{T^\sharp \leq t\}\frac{ |\xi_t^r-\xi_t^*|}{t-1} \right] 
\nonumber\\
= &\  B_{\xi} \E\left[\sum_{t=2}^T \vct{1}\{T^\sharp \leq t\}
\frac{ \E [|\xi_t^r-\xi_t^*|\mid \mathcal{F}_{t+1}]}{t-1} \right]
\leq B_{\xi} \E\left[\sum_{t=2}^T \vct{1}\{T^\sharp \leq t\}
\frac{ \sqrt{\E[|\xi_t^r-\xi_t^*|^2\mid \mathcal{F}_{t+1}]}}{t-1} \right],
\label{eq:C-intermediate-1}
\end{align}
where the last equality holds because the event $\{T^\sharp \leq t\} = \{T^\sharp \geq t+1\}^c \in \mathcal{F}_{t+1}$.

Because the regret is defined as the \emph{difference} between the expected revenues under the optimal DP policy and the re-solving heuristic, we can choose the
joint distribution of $(\xi_t^*, \xi_t^r)$ freely, as long as their marginal distributions remain the same.
We choose the joint distributions as follows.
At each time period $t$ with posted prices $p_t^*,p_t^r$ and corresponding demand rates $x_t^*+\Delta_t$ and $x_t^r$,
let $(\xi_t^*,\xi_t^r)\sim \Xi_t$
such that the marginal distributions are $Q_t(p_t^*)$, $Q_t(p_t^r)$ and furthermore 
\[
\sqrt{\E\left[|\xi_t^*-\xi_t^r|^2|\mathcal F_{t+1}\right]}
= \sqrt{\E_{(\xi_t^*,\xi_t^r)\sim\Xi_t}\left[|\xi_t^*-\xi_t^r|^2\right]}
\leq L|f(p_t^*)-f(p_t^r)| = L|x_t^*+\Delta_t-x_t^r|\quad a.s.
\]
The existence of such a joint distribution $\Xi_t$ is implied by $\mathcal W_2(Q(p_t^*),Q(p_t^r))\leq L|f(p_t^*)-f(p_t^r)|$ (see  Assumption (A5)).
As a result,
Eq.~(\ref{eq:C-intermediate-1}) can be simplified to
\begin{equation}
\big|\E[\mathcal C]\big|
\leq LB_\xi \E\left[\sum_{\tau=T^\sharp}^T \frac{|\Delta_{\tau}-\bar\Delta_{\to \tau}+\bar\xi_{\to \tau}^\delta|}{\tau-1}\right]+
\E\left[(T^\sharp-1)\eta^r\big|\bar\xi^{\delta}_{\to T^\sharp-1}\big|\right].
    \label{eq:C-final}
\end{equation}
Combining Eq.~(\ref{eq:proof-main-ub-3}) with Eqs.~(\ref{eq:A-final},\ref{eq:B-final},\ref{eq:C-final}), we have
\begin{align}
&\ \E\big[r'(x_T)\mathcal A - r''(x_T)\mathcal B + r''(x_T)\mathcal C + \mathcal D \big] \nonumber\\
\leq &\  \E\bigg[ (T^\sharp-1)\left(
r'(x_T)\eta - r''(x_T)(\bar\xi_{\to T^\sharp-1}^r+\eta^r)\eta 
-r''(x_T)\eta^r \big|\bar\xi^{\delta}_{\to T^\sharp-1}\big|
+ \frac{M^2}{8m}|\bar\xi_{\to T^\sharp-1}^r|^4
\right) \nonumber\\
&\  + \sum_{\tau=T^{\sharp}}^T \left(
\left(|r''(x_T)|\frac{LB_\xi}{\tau-1} + \frac{M}{2}|\bar\xi_{\to\tau}^r|^2\right)
|\Delta_\tau-\bar\Delta_{\to\tau}+\bar\xi_{\to\tau}^\delta| - \frac{m}{2}|\Delta_\tau-\bar\Delta_{\to\tau}+\bar\xi_{\to\tau}^\delta|^2\right)  
\bigg] \nonumber \\
\leq &\  E\bigg[ (T^\sharp-1)\left(
r'(x_T)\eta - r''(x_T)(\bar\xi_{\to T^\sharp-1}^r+\eta^r)\eta 
-r''(x_T)\eta^r \big|\bar\xi^{\delta}_{\to T^\sharp-1}\big|
+ \frac{M^2}{8m}|\bar\xi_{\to T^\sharp-1}^r|^4
\right) \nonumber \\
&\ +  \sum_{\tau=T^{\sharp}}^T \frac{1}{2m}
\left(|r''(x_T)|\frac{LB_\xi}{\tau-1} + \frac{M}{2}|\bar\xi_{\to\tau}^r|^2\right)^2 
 \bigg] \nonumber \\
\leq &\  E\bigg[ (T^\sharp-1)\left(
r'(x_T)\eta - r''(x_T)(\bar\xi_{\to T^\sharp-1}^r+\eta^r)\eta 
-r''(x_T)\eta^r \big|\bar\xi^{\delta}_{\to T^\sharp-1}\big|
+ \frac{M^2}{8m}|\bar\xi_{\to T^\sharp-1}^r|^4
\right) \nonumber\\
&\ +  \sum_{\tau=T^{\sharp}}^T \frac{1}{m}
\left(\bigl(|r''(x_T)|\frac{LB_\xi}{\tau-1}\bigr)^2 + \frac{M^2}{4}|\bar\xi_{\to\tau}^r|^4\right) 
 \bigg]. 
\label{eq:proof-main-ub-4}
\end{align}
In the second inequality above, we use the fact that
$-\frac{m}{2}u^2+bu = -\frac{m}{2}(u-\frac{b}{m})^2 +\frac{b^2}{2m} \leq \frac{b^2}{2m}$ 
for any $b,u\in\mathbb R$ and $m>0$.
In the third inequality, we use the fact that $(a+b)^2\leq 2(a^2+b^2)$ for any $a,b\in\mathbb{R}$.

Recall that $\eta^r=(r'(x_T)/r''(x_T)-\bar\xi_{\to T^\sharp-1}^r)^+$, so
$r'(x_T)-r''(x_T)(\bar\xi_{\to T^\sharp-1}^r+\eta^r) \geq 0$ and
$[r'(x_T)-r''(x_T)(\bar\xi_{\to T^\sharp-1}^r+\eta^r)]\eta^r = 0$.
Recall that $\eta^{*} = - (x^*_{T^\sharp -1}-x^{u})^+ \leq 0$,
so $[r'(x_T) - r''(x_T)(\bar\xi_{\to T^\sharp-1}^r+\eta^r)]\eta^* \leq 0$.
In sum, 
as $\eta = \eta^* + \eta^r$, we have
$[r'(x_T)-r''(x_T)(\bar\xi_{\to T^\sharp-1}^r+\eta^r)]\eta \leq 0$.

Furthermore, by the definition of the stopping time $T^\sharp$, we have $|\bar\xi_{\to T^\sharp-1}^r| = |\bar\xi_{\to T^\sharp}^r+\xi^r_{T^\sharp}/(T^\sharp-1)|\leq r'(x_T)/|r''(x_T)| + B_\xi/(T^\sharp-1)$, so $\eta^r \leq B_\xi/(T^\sharp-1)\ a.s$. 
By Eq.~\eqref{eq:proof-main-ub-4}, we get
\begin{align}
& \E\big[r'(x_T)\mathcal A - r''(x_T)\mathcal B + r''(x_T)\mathcal C + \mathcal D \big]\nonumber\\
\leq &
\E\bigg[ |r''(x_T)|B_\xi \big|\bar\xi^{\delta}_{\to T^\sharp-1}\big|  + (T^\sharp-1)\frac{M^2}{8m}|\bar\xi_{\to T^\sharp-1}^r|^4 +  \sum_{\tau=T^{\sharp}}^T \frac{1}{m}
\left(\bigl(|r''(x_T)|\frac{LB_\xi}{\tau-1}\bigr)^2 + \frac{M^2}{4}|\bar\xi_{\to\tau}^r|^4\right) 
 \bigg]. \label{eq:proof-main-ub-5}
\end{align}

To complete the proof, we upper bound each term in Eq.~(\ref{eq:proof-main-ub-5}).
First it is easy to verify that
\begin{equation}
\sum_{\tau=T^\sharp}^T \left(\frac{|r''(x_T)|LB_\xi}{\tau-1}\right)^2 \leq (|r''(x_T)|LB_\xi)^2 \sum_{j=1}^{T-1}\frac{1}{j^2} \leq 2(|r''(x_T)|LB_\xi)^2\quad a.s.
\label{eq:proof-main-ub-6}
\end{equation}
We next focus on the terms involving $|\bar\xi_{\to\tau}^r|^4$. 
Recall the definition that $\bar\xi_{\to t}^r = \frac{\xi_T^r}{T-1}+\cdots+\frac{\xi_{t+1}^r}{t}$.
Let $z_t:=\bar\xi_{\to t-1}^r$. Then $\{z_t\}$ is a martingale adapted to the filtration $\{\mathcal{F}_{t}\}_{t=1}^T$ by Assumption (A4), which implies that $\{|z_t|^4\}$ is a submartingale. Let $S_t = \sum_{\tau=t}^{T} (t-1)(|z_\tau|^4-|z_{\tau+1}|^4)$, then $\{S_t\}$ is also a submartingale. Since $T^\sharp$ is stopping time, we have
\[
\E\left[(T^\sharp-1)|\bar\xi_{\to T^\sharp-1}^r|^4 +  \sum_{\tau=T^{\sharp}}^T   |\bar\xi_{\to\tau}^r|^4 \right]
= \E[S_{T^\sharp}]\leq \E[S_{1}]=
\E\left[\sum_{\tau=2}^T |z_\tau|^4\right]
=
\E\left[\sum_{\tau=1}^T   |\bar\xi_{\to\tau}^r|^4 \right].
\]}
It is easy to verify that
\begin{align*}
\E[|\bar\xi_{\to t}^r|^4]
&= \sum_{j,k>t}\frac{\E[|\xi_j^r|^2|\xi_k^r|^2]}{(j-1)^2(k-1)^2}
\leq B_\xi^4 \left(\sum_{j>t}\frac{1}{(j-1)^2}\right)^2 \leq \frac{4 B_\xi^4}{t^2}.
\end{align*}
Subsequently, 
\begin{align}
\E\left[\frac{M^2}{8m}(T^\sharp -1)|\bar\xi_{\to T^\sharp-1}^r|^4 +\frac{M^2}{4m} \sum_{\tau=T^\sharp}^T\big|\bar\xi_{\to\tau}^r\bigl|^4\right]
&\leq \frac{M^2}{4m}\mathbb \sum_{\tau=1}^T \frac{4B_\xi^4}{\tau^2} \leq \frac{M^2B_\xi^4}{2m}.
    \label{eq:proof-main-ub-7}
\end{align}
To analyze the term $|\overline\xi_{\to T^\sharp -1}^\delta|$, let $\zeta_t:=\bar\xi_{t-1}^{\delta}$. As before, $\{\zeta_t\}$ is martingale and thus $\{|\zeta_t|\}$ is a submartingale. 
Since $\overline\xi_{\to\tau}^\delta = \frac{\xi_T^\delta}{T-1}+\cdots+\frac{\xi_{\tau+1}^\delta}{\tau}$ and $|\xi_{\tau}^{\delta}| = |\xi_{\tau}^{r}-\xi_{\tau}^{*}| \leq 2B_\xi\ a.s.$,
we have
\begin{equation}
    \E\big[ \big|\bar\xi^{\delta}_{\to T^\sharp-1}\big| \big] = \E[|\zeta_{T^\sharp}|]
    \leq \mathbb[|\zeta_{2}|] = \E\big[ \big|\bar\xi^{\delta}_{\to 1}\big| \big] \leq 
    \sqrt{\E\big[ \big|\bar\xi^{\delta}_{\to 1}\big|^2 \big]} \leq 2B_\xi \sqrt{\sum_{j=1}^{T-1}\frac{1}{j^2}} \leq {3B_\xi}.
    \label{eq:proof-main-ub-8}
\end{equation}


Combining Eqs.~(\ref{eq:proof-main-ub-3},\ref{eq:proof-main-ub-5},\ref{eq:proof-main-ub-6},\ref{eq:proof-main-ub-7},\ref{eq:proof-main-ub-8}) we have
\begin{align}
{
\phi_T^*(x_T)-\phi_T^r(x_T) \leq \frac{M^2B_\xi^4}{2m} + \frac{2(|r''(x_T)|LB_\xi)^2}{m} +
3|r''(x_T)|B_\xi^2 + r(x^u)\E[T^\sharp-1] = O(1),
}
\label{eq:const-regret-expression}
\end{align}
where the last equality holds by applying Lemma~\ref{lem:T-sharp-ub}.
This completes the proof of Theorem \ref{thm:main-ub}.

\subsection{Proof of Theorem \ref{thm:main-lb}}

Recall that $x_\tau^r = x_T-\bar\xi_{\to\tau}^r$ for all $\tau\geq T^\sharp-1$, where $T^\sharp$ is the stopping time defined in Eq.~\eqref{eq:defn-Tsharp}.
Because this proof only concerns the re-solving policy, for convenience we will drop the superscript $r$ and denote $\xi_\tau^r,\bar\xi_{\to\tau}^r$ by $\xi_\tau,\bar\xi_{\to\tau}$.

Invoking Lemma~\ref{lem:ft-star-expansion}, we have
\begin{align}
 Tr(x_T) -\phi_T^r(x_T)
& \geq\textstyle \E\left[\sum_{\tau=T^\sharp}^T (r(x_T) - r(x_\tau^r)) + (T^\sharp-1)(r(x_T)-r(\min\{x^r_{T^\sharp-1},x^u\})\right]\nonumber\\
& =\textstyle \E\left[\sum_{\tau=T^\sharp}^T (r(x_T) - r(x_\tau^r)) + (T^\sharp-1)(r(x_T)-r(x^r_{T^\sharp-1})\right] \nonumber\\
& \quad + \E\left[(T^\sharp-1)(r(x^r_{T^\sharp-1})-r(\min\{x^r_{T^\sharp-1},x^u\}) \right].\label{eq:proof-main-lb-0}
\end{align}
{
By the smoothness of $r(d)$ (Assumption (A3)) and the fact that $x^r_{T^\sharp}\leq x^u$ (see Eq.~\eqref{eq:defn-Tsharp}), the second term in Eq.~\eqref{eq:proof-main-lb-0} is bounded by
\begin{align*}
    \E\left[(T^\sharp-1)(r(x^r_{T^\sharp-1})-r(\min\{x^r_{T^\sharp-1},x^u\}) \right] \geq&\ 
-C\E\left[(T^\sharp-1)(x^r_{T^\sharp-1} - x^u)^+ \right] \\
\geq&\  -C\E\left[(T^\sharp-1)|x^r_{T^\sharp-1} - x^r_{T^\sharp} | \right]\\
=&\  \textstyle -C\E\left[(T^\sharp-1)\left|\frac{\xi_{T^\sharp}}{T^\sharp -1}\right| \right]
\geq -CB_\xi,
\end{align*}
where the last inequality holds since $|\xi_{\tau}|\leq B_\xi\ a.s$ for all $\tau$. 

In the rest of the proof, we will bound the first term in Eq.~\eqref{eq:proof-main-lb-0}.
} Expanding the difference $r(x_\tau^r)-r(x_T)$ at $x_T$ and using the strict concavity of $r(d)$ (Assumption (A2)), we have
\[
r(x_T) - r(x_\tau^r) = r(x_T) - r(x_T - \bar\xi_{\to\tau}) \geq r'(x_T)\bar\xi_{\to\tau} + \frac{m}{2}|\bar\xi_{\to\tau}|^2,
\qquad \forall \tau \geq T^\sharp -1.
\]
Therefore,
\begin{align}
&\textstyle\quad  \E\left[\sum_{\tau=T^\sharp}^T (r(x_T) - r(x_\tau^r)) + (T^\sharp-1)(r(x_T)-r(x^r_{T^\sharp-1}))\right]\nonumber\\
& \textstyle \geq  r'(x_T)\E\left[\sum_{\tau=T^\sharp}^T\bar\xi_{\to\tau} + (T^\sharp-1) \bar\xi_{\to T^\sharp-1} \right]
+ \dfrac{m}{2} \E\left[\sum_{\tau=T^\sharp}^T |\bar\xi_{\to\tau}|^2 + (T^\sharp-1)|\bar\xi_{\to T^\sharp-1}|^2\right]
\label{eq:proof-main-lb-1}
\end{align}
For the first term in Eq.~\eqref{eq:proof-main-lb-1}, because $\bar\xi_{\tau}=\sum_{t=\tau+1}^T \xi_t/(t-1)$, it holds that
\begin{equation}
\textstyle
\E[\sum_{\tau=T^\sharp}^T\bar\xi_{\to\tau} + (T^\sharp-1)\bar\xi_{\to T^\sharp-1} ] = \E[\sum_{t=T^\sharp}^T\xi_t] = 0
\label{eq:proof-main-lb-2}
\end{equation}
by Doob's optional stopping theorem (recall that $T^\sharp$ is a stopping time).
{ For the second term in Eq.~\eqref{eq:proof-main-lb-1}, 
using $\bar\xi_\tau = \frac{\xi_T}{T-1}+\cdots+\frac{\xi_{\tau+1}}{\tau}$,
we have
\begin{align}
   &\quad \E\left[\sum_{\tau=1}^T |\bar\xi_{\to\tau}|^2 \right]
    -  \E\left[\sum_{\tau=T^\sharp}^T |\bar\xi_{\to\tau}|^2 + (T^\sharp-1)|\bar\xi_{\to T^\sharp-1}|^2\right] = 
    \E\left[\sum_{\tau=1}^{T^\sharp-2} \left(\bar\xi_{\to\tau}^2-\bar\xi_{\to T^\sharp -1}^2\right) \right]\nonumber\\
    & 
    = 2\E\left[ \sum_{\tau=1}^{T^\sharp-2} 
    (\bar\xi_{\to\tau} - \bar\xi_{\to T^\sharp -1}) \bar\xi_{\to T^\sharp -1}\right]
    + \E\left[ \sum_{\tau=1}^{T^\sharp-2} 
    (\bar\xi_{\to\tau} - \bar\xi_{\to T^\sharp -1})^2\right]
    \nonumber\\
    & = 2  \E\left[ \bar\xi_{T^\sharp-1} \sum_{\tau=1}^{T^\sharp-2} \sum_{j = \tau+1}^{T^\sharp-1}\frac{\xi_j}{j-1} \right]
    + \E\left[ \sum_{\tau=1}^{T^\sharp-2} \left(\sum_{j = \tau+1}^{T^\sharp-1}\frac{\xi_j}{j-1} \right)^2\right].
    \label{eq:proof-main-lb-3}
\end{align}
Because $T^\sharp$ is a stopping time and $\{\xi_\tau\}$ is a martingale difference sequence, we have
$\E[{\xi_j}\mid j<T^\sharp]=0$ and
$\E[{\xi_j}{\xi_k}\mid j < k<T^\sharp]=0$. So we have
\[
\text{Eq.}\eqref{eq:proof-main-lb-3} = 
0 + \E\left[\sum_{\tau=1}^{T^\sharp-2} \sum_{j = \tau+1}^{T^\sharp-1}\left(\frac{\xi_j}{j-1} \right)^2\right]
= \E\left[\sum_{j=2}^{T^\sharp-1} (j-1)\left(\frac{\xi_j}{j-1} \right)^2\right]
\leq \E\left[\sum_{j=2}^{T^\sharp-1}\xi_{j-1}^2\right]\leq B_\xi^2\E[T^\sharp-2].
\]
Moreover, $\E[\sum_{\tau=1}^T \bar\xi_{\to\tau}^2] = \E[\sum_{j=2}^T \frac{\xi^2_j}{j-1}] \geq \sigma^2 \sum_{j=2}^T \frac{1}{j-1} \geq \sigma^2 \int_{1}^T u \mathrm{d}u
= \sigma^2 \ln T$, so by the above equation and Eq.~\eqref{eq:proof-main-lb-3}, we have
\begin{equation}
    \E\left[\sum_{\tau=T^\sharp}^T |\bar\xi_{\to\tau}|^2 + (T^\sharp-1)|\bar\xi_{\to T^\sharp-1}|^2\right] \geq \sigma^2 \ln T - 
    B_\xi^2\E[T^\sharp-2] = \sigma^2 \ln T - O(1),
    \label{eq:proof-main-lb-4}
\end{equation}
where the last equality holds by Lemma~\ref{lem:T-sharp-ub}.
}

Combining Eqs.~(\ref{eq:proof-main-lb-0},\ref{eq:proof-main-lb-1},\ref{eq:proof-main-lb-2},\ref{eq:proof-main-lb-4}), we get $Tr(x_T) -\phi_T^r(x_T) \geq \frac{m\sigma^2}{2}\ln T - O(1)$
and the proof of Theorem~\ref{thm:main-lb} is complete.

\section{Numerical Results}
\label{sec:numerical_results}

We corroborate the theoretical findings in this paper with a few simple numerical experiments.
In the simulation we assume a linear demand curve with Bernoulli demand distribution: $\Pr[d_t=1|p_t]=\alpha-\beta p_t$, $\Pr[d_t=0|p_t]=1-\Pr[d_t=1|p_t]$
with $p\in[0, 1]$, $\alpha=3/4$ and $\beta=1/2$.
The (normalized) initial inventory level is $x_T=5/16$, meaning that for problem instances with $T$ time periods
the initial inventory level is $x_TT =5T/16$.
It is easy to verify that the optimal demand rate $x^{u}$ without inventory constraints is $x^{u}=3/8>x_T$,
and the fluid approximation suggests a $Tr(x_T) = (19/32)T=.59375 T$ expected revenue.
We select the Bernoulli demand distribution because the states of inventory levels are discrete and therefore the optimal dynamic programming
pricing policy can be exactly obtained.

\begin{table}[hbt]
\centering
\caption{Regret for the fluid model, the optimal static policy $\pi^s$ and the re-solving heuristics $\pi^r$
compared against the value of the optimal DP pricing policy.}
\label{tab:result}
\begin{tabular}{l|cccccccccc}
\hline
$\log_2 T$& 6& 7& 8& 9& 10& 11& 12& 13& 14& 15\\
\hline
Fluid Model& -0.90& -1.13& -1.37& -1.63& -1.91& -2.19& -2.48& -2.78& -3.08& -3.37\\
Static policy $\pi^s$& 0.38& 0.70& 1.22& 2.03& 3.27& 5.13& 7.84& 11.81& 17.55& 25.84 \\
Resolving heuristics $\pi^r$&0.11& 0.15& 0.18& 0.21& 0.23& 0.23& 0.24& 0.24& 0.24& 0.25 \\
\hline
\end{tabular}
\end{table}

In Table \ref{tab:result} we report the regret of the fluid approximation, the static policy $\pi^s: p_t\equiv f^{-1}(x_T)=7/8$
and the re-solving heuristics $\pi^r$.
All regret is defined with respect to the value (expected reward) of the optimal DP pricing policy,
and the regret for the fluid approximation benchmark is negative since the fluid model always upper bounds the value of any policy.
Both the static policy $\pi^s$ and the re-solving heuristics $\pi^r$ are run for each value of $T$
ranging from $T=2^6=64$ to $T=2^{15}=32,768$ to 
obtain an accurate estimation of their expected rewards. 
We also plot the regret in Figure \ref{fig:result} to make the regret growth of each policy more intuitive.

\begin{figure}[tb]
\centering
\includegraphics[width=0.6\textwidth]{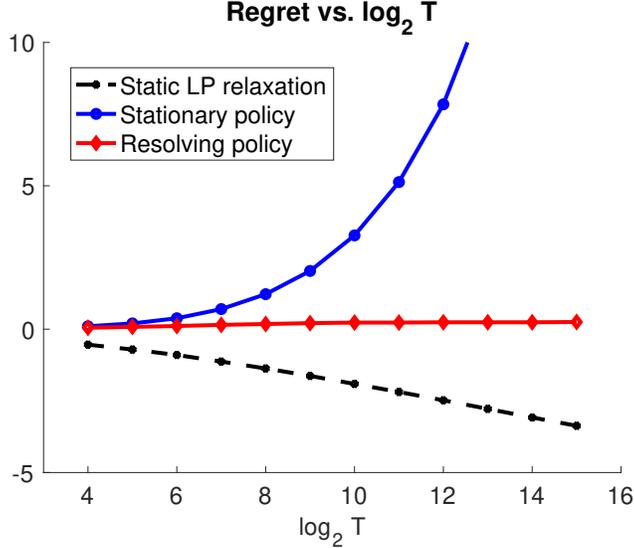}
\caption{Plots of regret of the fluid model, the static policy $\pi^s$ and the re-solving heuristics $\pi^r$
compared against the value of the optimal DP pricing policy.}
\label{fig:result}
\end{figure}

As we can see from Table \ref{tab:result}, the gap between the value of the optimal policy and
the value of the fluid model grows nearly linearly as the number of time periods $T$ grows geometrically,
which verifies the $\Omega(\ln T)$ growth rate established in Theorem \ref{thm:main-lb}.
On the other hand, the growth of regret of the re-solving heuristics $\pi^r$ stagnated at $T\geq 2^{10}$
and is nearly the same for $T$ ranging from $2^{10}=1024$ to $2^{15}=32768$.
This shows the asymptotic growth of regret of $\pi^r$ is far slower than $O(\ln T)$ and is compatible with the $O(1)$ regret upper bound
we proved in Theorem \ref{thm:main-ub}.

\begin{figure}[t]
\centering
\includegraphics[width=0.48\textwidth]{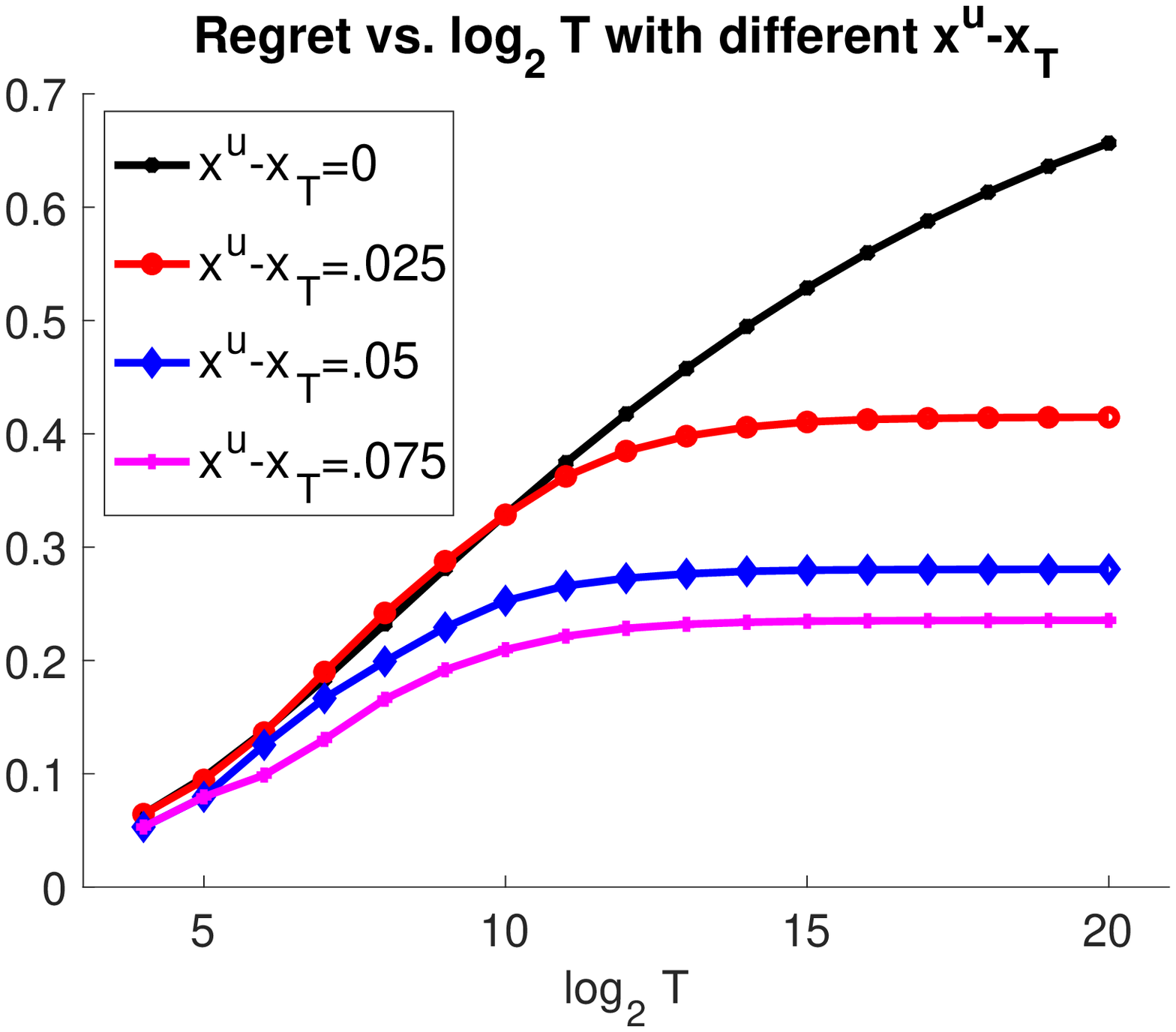}
\includegraphics[width=0.48\textwidth]{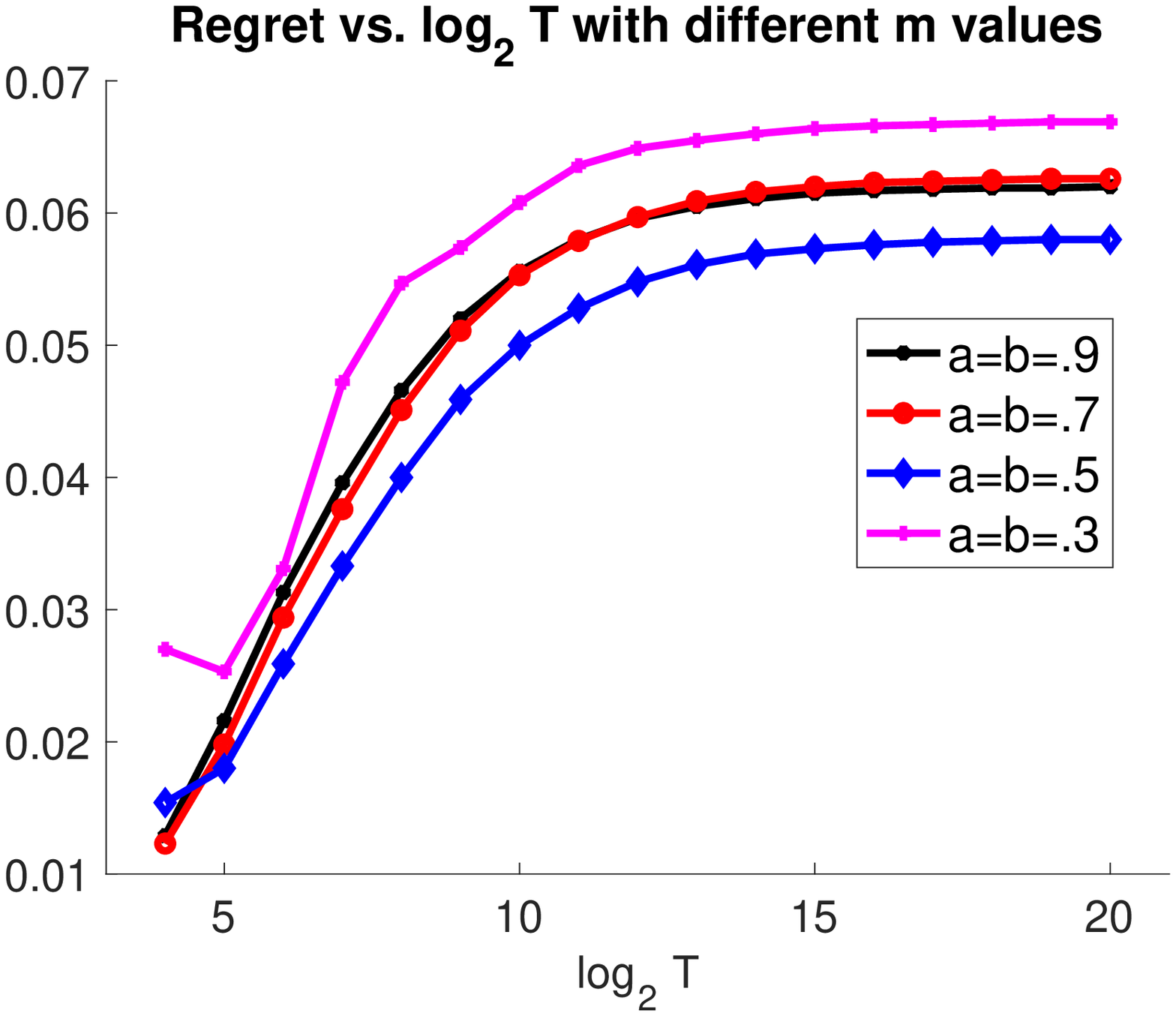}
\label{fig:regret-gap-m}
\caption{Cumulative regret of the resolving heuristic (compared with the optimal DP value) under different $x^u-x_T$ and $m$ settings.}
\end{figure}

{
We report additional sets of numerical results in Figure \ref{fig:regret-gap-m}, in which we only report the cumulative regret of the re-solving heuristic compared against the benchmark of the optimal DP policy.

On the left panel of Figure \ref{fig:regret-gap-m}, we report the regret of the re-solving heuristic with $T$ ranging from $2^4=16$ to $2^{20}\approx 1,000,000$ and different $x^u-x_T$ gap values. More specifically, all four curves are reported under the demand model $d=.75-.5p$, with unconstrained optimum $x^u=.375$ and (normalized) initial inventory levels $x_T\in\{.3, .325, .35, .375\}$.
Figure \ref{fig:regret-gap-m} clearly shows that the regret of the re-solving heuristic increases as the gap between $x^u$ and $x_T$ narrows, and furthermore in the boundary case (i.e., $x_T=x^u$) the regret seems to grow logarithmically in $T$.
It is an interesting direction of future research to formally establish the logarithmic regret for the boundary case and explore alternative policies that attain constant regret with $x_T=x^u$.

On the right panel of Figure \ref{fig:regret-gap-m}, we report the regret of the re-solving heuristic with different $a,b$ values in the demand model $d=a-bp$, with $a=b\in\{0.3,0.5,0.7,0.9\}$. The normalized initial inventory level is fixed at $x_T=0.1$. With different values of the slopes, the demand and revenue models exhibit different strong concavity parameter values, with $r''(p)=-b/2$ and $r''(d) = -1/(2b)$. Unlike the gap $x^u-x_T$, the results reported in the right panel of Figure \ref{fig:regret-gap-m} do not paint a clear picture of the role the strong concavity parameters played in the regret. Overall, intermediate values ($b=0.5, r''(p)=-.25, r''(d)=-1$) seem to result in the lowest regret of the re-solving heuristic.
}

\section{Extension to Multiple Products}\label{sec:multi-product}

In this section, we extend our constant regret result to the case when there are $n\geq 2$ products with correlated demand. 
We follow the convention in Section~\ref{sec:proofs} and count time indices backwards.
The seller starts with an initial inventory vector $\BFx_T T$, where $\BFx_T=(x_{T}(1),\cdots,x_{T}(n))\in\mathbb R_+^n$ is the normalized initial inventory vector.
When there are $t$ periods remaining, the seller posts a price vector $\BFp_t=(p_{t}(1),\cdots,p_{t}(n))\in \mathcal{P} \subset R_+^n$ and observes a realized demand vector $\BFd_t = \BFf(\BFp_t)+\vct\xi_t$, where $\BFf: \mathcal{P} \to \mathcal{D}$ is the demand curve and $\BFxi_t\sim Q(\BFp_t)$ is a centered noise vector. The realized revenue at period $t$ is $r_t = \langle{\BFp_t},{\BFd_t}\rangle$, which in expectation equals to $\E[r_t \mid \BFp_t] = \langle{\BFp_t},{\BFf(\BFp_t)}\rangle$.
When the inventory level of a specific product $k\in[n]$ dips below zero ($[n]:=\{1,2,\cdots,n\}$), the seller is forced to set $p_{t}(k)=+\infty$ for all remaining time periods, resulting in $d_{t}(k)=0$ almost surely.

Throughout this section, we use $\|\BFx\|$ to denote the Euclidean norm of any vector $\BFx$, and $\|\BFA\|_2$ to denote the spectral norm of any matrix $\BFA$.
We extend the assumptions (A1)--(A5) to the multiple-product setting as follows:
\begin{enumerate}
\item[B1.] {(\bf Invertibility)} The demand rate function $\BFf: \mathcal{P} \to \mathcal{D}$ is 
a bijection, where $\mathcal{P}\subset \mathbb{R}^n_+, \mathcal{D} \subset \mathbb{R}^n_+$. Let $\BFf^{-1}: \mathcal{D} \to \mathcal{P}$ denote its inverse function. Assume $\mathcal{D}$ is convex, compact, has nonempty interior, and satisfies $\BFzero \in \mathcal{D}$.
\item[B2.] {(\bf Strict Concavity)} The expected revenue $r(\BFd)=\langle \BFd, \BFf^{-1}(\BFd)\rangle$ as a function of the demand rate vector $\BFd$ is strictly concave. That is, there exists a positive constant $m'>0$ such that $\nabla^2 r(\BFd)\preceq -m'\mat I_n$ for all $\BFd\in \mathcal{D}$.
The maximizer of $r(\BFd)$ is in the interior of the domain $\mathcal{D}$.
\item[B3.] {(\bf Smoothness)} $r(\BFd)$ is three times continuously differentiable in the interior of $\mathcal{D}$.
\item[B4.] {(\bf Martingale Difference Sequence)}
Conditional on any price vector $\BFp_t \in \mathcal{P}$ ($\forall t = 1,\ldots,T$), the demand noise $\vct\xi_t$ is independent of $\{\vct\xi_{T},\ldots,\vct\xi_{t+1}\}$ and satisfies 
$\E[\vct\xi_t \mid \mathcal{F}_{t+1}] = \E[\vct\xi_t \mid \BFp_t] =\BFzero \ a.s.$ 
The conditional distribution of $\vct\xi_t$ given $\BFp_t$ is denoted by $\vct\xi_t\sim Q(\BFp_t)$. In addition, $\|\vct\xi_t\|_{}\leq B_\xi \ a.s.$ for some constant $B_\xi<\infty$.
\item[B5.] {\bf (Wasserstein Distance)} There exists a constant $L>0$ such that for any $\BFp,\BFp'\in \mathcal{P}$, it holds that $\mathcal W_2(Q(\BFp),Q(\BFp'))\leq L\|\BFf(\BFp)-\BFf(\BFp')\|_{}$,
where $\mathcal W_2(Q,Q'):= \inf_{\Xi}\sqrt{\E_{\vct\xi,\vct\xi'\sim\Xi}[\|\vct\xi-\vct\xi'\|_{}^2]}$ is the $L_2$-Wasserstein distance
between $Q,Q'$, with $\Xi$ being an arbitrary joint distribution with marginal distributions being $Q$ and $Q'$, respectively.
\end{enumerate}

Recall that $\BFx_T \in \mathbb{R}^n_+$ is the normalized inventory level vector at the beginning of the $T$ time periods (i.e., the total inventory is $\BFx_T T$). The fluid approximation model is formulated as
\[
    \max_{\BFx \in \mathcal{D}} \left\{ r(\BFx),\ 
    \mathrm{s.t.}\ \BFzero \leq \BFx  \leq \BFx_T\right\}. 
\]
At the beginning of each period $t=T,\cdots,1$, the re-solving heuristic solves 
\begin{align}\label{eq:fluid-multiple-product}
    \max_{\BFx \in \mathcal{D}} \{r(\BFx), 
    \ \mathrm{s.t.}\  \BFzero \leq  \BFx \leq \BFx_t\},
\end{align}
where $\BFx_t$ is the normalized inventory level at the beginning of period $t$. 
Since $\BFzero \in \mathcal{D}$ by Assumption (B1),  Eq.~\eqref{eq:fluid-multiple-product} always has feasible solutions.
Let the optimal solution to \eqref{eq:fluid-multiple-product} be $\BFx^c_t$ (the superscript $c$ stands for ``constrained"). The re-solving policy sets the price vector to $\BFp_t= \BFf^{-1}(\BFx^c_t)$ for period $t$.

In \eqref{eq:fluid-multiple-product}, for any given $\BFx_t \in \mathbb{R}^n_+$,
the $n$ products are partitioned into two disjoint sets based whether the inventory constraints are active at the point $\BFx^c_t$. The \emph{inventory-constrained} product $\mathcal I$ and the \emph{inventory-unconstrained} product set $\mathcal U$, defined as
\begin{align}
    \mathcal I &:= \{k\in[n]: \BFx^c_t(k)=\BFx_t(k)\}, \;\;\;\;\;\;
    \mathcal U := \{k\in[n]: \BFx^c_t(k)<\BFx_t(k)\}.
    \label{eq:set-partition}
\end{align}
As a special case, in the single-product setting ($n=1$), we have $\mathcal I=\varnothing,\  \mathcal{U} = \{1\}$ if $x_t > x^u = \arg\max_{x\in[\underline d,\overline d]}r(x)$, and $\mathcal I = \{1\},\ \mathcal U=\varnothing$ if $x_t \leq x^u$. 

From the definitions in Eqs.~(\ref{eq:fluid-multiple-product}, \ref{eq:set-partition}), it is clear that the sets $\mathcal I$ and $\mathcal U$ are determined by the inventory vector $\BFx_t$, which serves as the right-hand side of the fluid problem. We may thus write $\mathcal I(\BFx_t),\ \mathcal U(\BFx_t)$ to emphasize such dependency. This leads to a partition of $\mathbb{R}^n_+$ into $2^n$ sub-regions $\{\mathcal{S}_{\mathcal{I}},\ \forall \mathcal{I} \subset [n]\}$, where 
\begin{equation}\label{eq:defn-SI}
    \mathcal{S}_{\mathcal{I}} = \{\BFx \in \mathbb{R}^n_+:\ \mathcal I(\BFx) = \mathcal{I}\}.
\end{equation}
We specify an additional assumption on the initial inventory level $\BFx_T$ for the multi-product setting:
\begin{enumerate}
    \item [C1.] 
    The initial inventory level $\BFx_T$ is in the interior of $\mathcal{S}_{\mathcal{I}}$ for some $\mathcal{I}\subset [n]$. That is, given
    $\BFx_T \in \mathcal{S}_\mathcal{I}$, there exists 
    a neighborhood $B_{\delta_0}(\BFx_T)=
    \{\BFx'\in\mathbb{R}^n_+: \|\BFx'-\BFx\|_{} \leq \delta_0\}$ such that $B_{\delta_0}(\BFx_T) \subset \mathcal{S}_\mathcal{I}$.
\end{enumerate}
Intuitively, Assumption (C1) asserts that when the constrained inventory level $\BFx_t$ fluctuates in a close neighborhood of $\BFx_T$, the set of active constraints in the fluid problem Eq.~(\ref{eq:fluid-multiple-product}) remains unchanged.
In the single-product case ($n=1$), Assumption (C1) reduces to the condition $x_T\neq x^u$.

We now characterize the expected value of the optimal DP policy $\pi^*$ and the re-solving heuristic $\pi^r$ for the multi-product pricing problem. Let $\phi_\tau^*(\BFx_\tau^*)$ be the expected revenue of the optimal policy $\pi^*$ when there are $\tau$ time periods left with the inventory level being $\BFx_\tau^*\in\mathbb R_+^n$. We have the following Bellman equation:
\begin{equation}
\phi_\tau^*(\BFx_\tau^*) = \max_{\vct\Delta: \BFx_\tau^*+\vct\Delta\in \mathcal{D}} r(\BFx_\tau^*+\vct\Delta) + \E_{\vct\xi_\tau\sim Q(\BFf^{-1}(\BFx_\tau^*+\vct\Delta))}\left[\phi_{\tau-1}^*\left(\BFx_\tau^*-\frac{\vct\Delta+\vct\xi_\tau}{\tau-1}\right)\right]\quad \forall \tau = T,\cdots,1.
    \label{eq:bellman-star-multi-product}
\end{equation}
We denote the maximizer of Eq.~(\ref{eq:bellman-star-multi-product}) by $\vct\Delta_\tau$.
The DP policy selects the price vector $\BFp_\tau=\BFf^{-1}(\BFx_\tau^*+\vct\Delta_\tau)$. Let $\BFxi^*_{\tau}$ be the realized demand noise under this price.

The re-solving heuristic $\pi^r$ and its expected revenue $\phi_\tau^r(\BFx_\tau^r)$, on the other hand, satisfies the following recursive equation
\begin{equation}
\phi_\tau^r(\BFx_\tau^r) = r(\BFx_\tau^c) + \E_{\vct\xi_\tau\sim Q(\BFf^{-1}(\BFx_\tau^c))}\left[\phi_{\tau-1}^r\left(\BFx_\tau^r - \frac{\BFx_\tau^c-\BFx_\tau^r+\vct\xi_\tau}{\tau-1
}\right)\right]\quad \forall \tau = T,\cdots,1,
    \label{eq:bellman-r-multi-product}
\end{equation}
where $\BFx_\tau^c = \arg\max_{\BFx \in \mathcal{D}}\{r(\BFx),\ \mathrm{s.t.}\ \BFzero \leq \BFx \leq \BFx^r_\tau\}$ is the solution to the fluid model with $\BFx_\tau^r$ being the right-hand side.
The following theorem extends the constant regret result in Sec.~\ref{sec:main-results} to the multi-product setting.
\begin{theorem}\label{thm:upper-bound-multi-product}
Given a demand function $\BFf$ and an initial inventory level $\BFx_T\in \mathcal{S}_{\mathcal{I}}$ satisfying Assumptions (B1)--(B5) and (C1), for all $T\geq 2$, we have 
\[
\phi_T^r(\BFx_T)\geq \phi_T^*(\BFx_T)-O(1).
\]
\end{theorem}
In the rest of this section we prove Theorem \ref{thm:upper-bound-multi-product}.

\subsection{Partial Optimization on a Subset of Products}

For any vector $\BFx\in\mathbb R^n$ and subset $\mathcal S\subset [n]$, denote $\BFx(\mathcal S)=(\BFx(k): k\in\mathcal S)$ as the $|\mathcal S|$-dimensional sub-vector of $\BFx$ whose coordinates are restricted to the subset $\mathcal S$.
The following observation follows immediately from the definition of $\mathcal{S}_{\mathcal{I}}$ and the proof is omitted.
\begin{lemma}\label{lem:sub-regions}
Suppose $\BFx \in \mathcal{S}_{\mathcal{I}}$ for some $\mathcal{I} \subset [n]$. 
Let $\mathcal{U} = [n] \setminus \mathcal{I}$.
For any inventory vector $\BFx' \in \mathbb{R}^n_+$ with $\BFx'(\mathcal{I}) = \BFx(\mathcal{I})$ and 
$\BFx(\mathcal{U})' \geq \BFx(\mathcal{U})$, we have $\BFx' \in \mathcal{S}_{\mathcal{I}}$.
\end{lemma}

Given a fixed $\BFx_T\in \mathcal{S}_{\mathcal{I}}$
with $\mathcal I\subset [n]$, let $\mathrm{proj}_{\mathcal I}(\mathcal{D})$
be the projection of the domain $\mathcal{D}$ onto the dimensions in $\mathcal{I}$.
For every $\BFz\in \mathrm{proj}_{\mathcal I}(\mathcal{D})$, define $R(\BFz)$ as the value of partial optimization of $r(\BFx)$ by changing the variables in $\mathcal{U} = [n] \setminus \mathcal{I}$ while fixing the variables in $\mathcal{I}$:
\begin{equation}
R(\BFz) := \max_{\BFx \in \mathcal{D}} \{r(\BFx),\ \mathrm{s.t.}\ \BFx(\mathcal I) = \BFz\}.
\label{eq:defn-R}
\end{equation}
The motivation for Eq.~(\ref{eq:defn-R}) is that in the region $\mathcal{S}_{\mathcal{I}}$, the solution of the re-solving heuristic is only affected by the inventory levels of the \emph{constrained products} in $\mathcal{I}$.
The following lemma establishes some useful properties of the function $R$.
\begin{lemma}
The function $R$ has the following properties:
\begin{enumerate}
    \item For any $\BFx_t \in \mathcal{S}_{\mathcal{I}}$,
    let $\BFz_t = \BFx(\mathcal{I})$.
    It holds that
    $R(\BFz_t) = \max_{\BFx \in \mathcal{D}}\{r(\BFx),\ \mathrm{s.t.}\ \BFzero \leq \BFx \leq \BFx_t \} = r(\BFx^c_t)$. 
    \item $R(\BFz)$ is strictly concave and three times continuously differentiable for all $\BFz\in \mathrm{proj}_{\mathcal{I}}(\mathcal{D})$
and satisfies 
$\nabla^2 R(\BFz)\preceq -m\mat I$, $\|\nabla^3 R(\BFz)\|_{\op}\leq M$ with some constants $m>0, M>0$.
\item Given $\BFz, \BFz'\in \mathrm{proj}_{\mathcal{I}}(\mathcal{D})$, let 
$\BFx, \BFx'$ be the optimal solutions in \eqref{eq:defn-R}.
Then $\|\BFx - \BFx'\|_{} \leq L_z \|\BFz - \BFz'\|_{}$ 
for some constant $L_z>0$
uniformly on $\mathrm{proj}_{\mathcal{I}}(\mathcal{D})$.
\end{enumerate}
\label{lem:properties-R}
\end{lemma}
{
\begin{proof}{Proof of Lemma \ref{lem:properties-R}.}
The first property follows immediately from the definition of 
$R(\BFz)$ in Eq.~\eqref{eq:defn-R} and the definition of 
$\mathcal{S}_{\mathcal{I}}$ in Eq.~\eqref{eq:defn-SI}.

To prove the second property,
consider the Lagrangian $\mathcal{L}(\BFx, \BFlambda, \BFz) := r(\BFx) + \BFlambda^\top (\BFz - \BFx(\mathcal{I}))$ of \eqref{eq:defn-R}.
Since $r(\BFx)$ is concave, $\mathcal{D}$ is a convex set with nonempty interior, and the optimization problem \eqref{eq:defn-R} satisfies Slater's condition, $\mathcal{L}$ has a saddle point $(\BFx^*(\BFz), \BFlambda^*(\BFz))$ satisfying
\[
R(\BFz) = \max_{\BFx \in \mathcal{D}} \min_{\BFlambda \in \R^{|\mathcal{I}|}} \mathcal{L}(\BFx, \BFlambda, \BFz) = \min_{\BFlambda \in \R^{|\mathcal{I}|}} \max_{\BFx \in \mathcal{D}} \mathcal{L}(\BFx, \BFlambda, \BFz) = \mathcal{L}(\BFx^*(\BFz), \BFlambda^*(\BFz), \BFz).
\]
Note that the saddle point $(\BFx^*(\BFz), \BFlambda^*(\BFz))$ is unique because $r(\BFx)$ is strictly convex.
By the envelop theorem for saddle point problems \citep[Theorem 5]{milgrom2002envelope}, when the saddle point is unique for every $\BFz$, the function $R(\BFz)$ is differentiable in the interior of $\mathrm{proj}_{\mathcal{I}}(\mathcal{D})$ with 
$
    \nabla R(\BFz) = \lambda^*(\BFz).
$
By the KKT conditions,  $\nabla_{\mathcal{I}}r(\BFx^*(\BFz)) = \lambda^*(\BFz)$, so
\begin{equation}\label{eq:KKT-1}
   \nabla R(\BFz) = \nabla_{\mathcal{I}}r(\BFx^*(\BFz)). 
\end{equation}
(Here $\nabla_{\mathcal{S}}r(\BFx)$ denotes the restriction of $\nabla r(\BFx)$ to a subset $\mathcal{S} \in [n]$.)

Denote the Hessian of $r(\BFx)$ at the point $\BFx^*(\BFz)$ by
\[
   \nabla^2 r(\BFx^*(\BFz)) = \mat H = \begin{bmatrix}
\mat H_{\mathcal{I}\times \mathcal{I}} & \mat H_{\mathcal{I}\times \mathcal{U}}\\
\mat H_{\mathcal{U}\times \mathcal{I}} & \mat H_{\mathcal{U}\times \mathcal{U}}.
\end{bmatrix}
\]
Let $\BFx^*(\BFz) = (\BFz, \BFu^*(\BFz))$.
By the KKT conditions again, we have $\nabla_{\mathcal{U}}r(\BFz, \BFu^*(\BFz))=\BFzero$. By the implicit function theorem,  $\BFu^*(\BFz)$ is continuously differentiable
and the Jacobian matrix of $\BFu^*(\BFz)$ is $\mat J_{\BFu}(\BFz) = - \mat H_{\mathcal{U}\times \mathcal{U}}^{-1}\mat H_{\mathcal{U}\times \mathcal{I}}$. Note that $\mat H_{\mathcal{U}\times \mathcal{U}}$ is invertible because $\mat H$ is negative definite by Assumption (B2). Using Eq.~\eqref{eq:KKT-1} and the chain rule, we have
\begin{equation}\label{eq:schur-complement}
\nabla^2 R(\BFz) = \nabla^2_{\mathcal{I}\times \mathcal{I}}r(\BFz, \BFu^*(\BFz)) + \nabla^2_{\mathcal{I}\times \mathcal{U}}r(\BFz, \BFu^*(\BFz))\cdot J_{\BFu}(\BFz)
= \mat{H}_{\mathcal{I}\times \mathcal{I}} - \mat H_{\mathcal{I}\times \mathcal{U}}
\mat{H}_{\mathcal{U}\times \mathcal{U}}^{-1}\mat H_{\mathcal{U}\times \mathcal{I}}. 
\end{equation}
The right-hand side of Eq.~\eqref{eq:schur-complement} is the Schur complement of the block $\mat{H}_{\mathcal{U}\times \mathcal{U}}$ in the matrix $\mat{H}$.
Since $\mat H$ is negative definite by Assumption (B2), the Schur complement is also negative definite \citep[Theorem 1.12]{zhang2006schur}, which implies that $R(\BFz)$ is strictly concave.

Next, we establish the smoothness condition of $R(\BFz)$.
Recall that $\nabla_{\mathcal{U}}r(\BFz, \BFu^*(\BFz))=\BFzero$, the Jacobian matrix $J_{\BFu}(\BFz)$ is invertible, and $\nabla r$ is twice continuously differentiable by Assumption (B3). By the implicit function theorem for $C^2$ class \citep[Theorem 3.3.1]{krantz2012implicit}, $\BFu^*(\BFz)$ is also twice continuously differentiable.
By Eq.~\eqref{eq:schur-complement}, $R(\BFz)$ is three times countinuously diffrentiable.
Because the domain $\mathcal{D}$ is compact by Assumption (B1), $\nabla^3 R(\BFz)$ is uniformly bounded in $\mathrm{proj}_{\mathcal{I}}(\mathcal{D})$.

Finally, the third statement of the lemma regarding the Lipschitz continuity of $\BFx^*(\BFz)$ is straightforward, because the domain $\mathrm{proj}_{\mathcal{I}}(\mathcal{D})$ is compact and we have shown that $\BFx^*(\BFz) = (\BFz, \BFu^*(\BFz))$ is continuously differentiable in $\BFz$.
$\square$

\end{proof}
}

\subsection{Stopping Time with Bounded Expectation}

Recall that $\{\vct\xi_\tau^r\}_{\tau=1}^T$ are the stochastic demand noise vectors at each time period on the path of the re-solving heuristic policy $\pi^r$.
For any $\tau$, define $\bar{\vct\xi}_{\to\tau}^r := \frac{\vct\xi_T^r}{T-1}+\cdots+\frac{\vct\xi_{\tau+1}^r}{\tau}$. 
Let $r_0 := \min_{k\in \mathcal{I}} \partial R(\BFx_T(\mathcal{I})) / \partial x_T(k)$.
Note that $r_0 >0$ by Assumption (C1).
Define $T^\sharp$ as
\begin{equation}
    T^\sharp := \max \left\{\tau \geq 1:\;\; \|\vct{\bar\xi}_{\to{\tau-1}}^r\|_{} > \min\big(\delta_0, r_0 / \|\nabla^2 R(\BFx_T(\mathcal{I})) \|_{2} \big) \right\} \vee 2,
    \label{eq:defn-Tsharp-multi-product}
\end{equation}
where $\delta_0>0$ is the constant parameter in Assumption (C1). Because $\bar\BFxi_{\to{t-1}}^r$ is measurable with respect to $F_t$, the event $\{T^\sharp=t\}$ is adaptive to the filtration $\{\mathcal F_t\}$ and therefore $T^\sharp$ is a stopping time.  We remark that $\|\vct{\bar\xi}_{\to{\tau-1}}^r\|_{} \leq r_0 / \|\nabla^2 R(\BFx_T(\mathcal{I})) \|_{2}$ implies the inequality 
$\nabla R(\BFx_T(\mathcal{I})) - \nabla^2 R(\BFx_T(\mathcal{I})) \bar\BFxi_{\to{\tau-1}}^r \geq 0$. We will use this fact later in the proof of Theorem~\ref{lem:T-sharp-multi-product}.

Recall that $\BFx_\tau^r$ is the normalized inventory vector when $\tau$ time periods are remaining. The following lemma gives characterization of $\BFx_\tau^r$ in terms of $\vct{\bar\xi}_{\to\tau}^r$. It also gives an upper bound on $\E[T^\sharp]$, similar to Lemma \ref{lem:T-sharp-ub}.
\begin{lemma}
Under Assumptions (B1)-(B5) and (C1), the following holds for all $\tau\geq T^\sharp-1$:
\begin{align}
    \BFx_\tau^r(\mathcal I) =&\ \BFx_T^r(\mathcal I) - \vct{\bar\xi}_{\to\tau}^r(\mathcal I); \label{eq:induction-xr}\\
    \BFx_\tau^r(\mathcal U) \geq&\ \BFx_T^r(\mathcal U) - \vct{\bar\xi}_{\to\tau}^r(\mathcal U),\nonumber
\end{align}
where $\mathcal I=\mathcal I(\BFx_T)$, $\mathcal U=\mathcal U(\BFx_T)$ as defined in Assumption (C1). For $\tau\geq T^\sharp$, we have $\BFx_\tau^r \in \mathcal{S}_\mathcal{I}$.  Furthermore, the stopping time is bounded by
$
\E[T^\sharp] = O(1).
$
\label{lem:T-sharp-multi-product}
\end{lemma}
\begin{proof}{Proof of Lemma \ref{lem:T-sharp-multi-product}.}
We first prove Eq.~\eqref{eq:induction-xr} by induction.
The base case of $\tau=T$ clearly holds because $\BFx_\tau = \BFx_T$, which belongs to $\mathcal{S}_\mathcal{I}$ by Assumption (C1).
Suppose Eq.~\eqref{eq:induction-xr} holds at period $\tau$. Because $\tau\geq T^\sharp$, it holds that $\|\bar{\vct\xi}_{\to\tau}^r\|_{}\leq\delta_0$ and therefore $\BFx_\tau^r\in \mathcal{S}_{\mathcal{I}}$ by Assumption (C1) and Lemma~\ref{lem:sub-regions}. Let $\BFx_\tau^{c} = \arg\max_{\BFx \in \mathcal{D}}\{r(\BFx),\ \mathrm{s.t.}\ \BFx \leq \BFx_\tau^r\}$ be the demand rate chosen by the re-solving heuristic at period $\tau$. Since $\BFx_\tau^r\in \mathcal{S}_{\mathcal{I}}$, we know that $\BFx_\tau^{c}(\mathcal I)=\BFx_\tau^r(\mathcal I)$ and $\BFx_\tau^{c}(\mathcal U)<\BFx_\tau^r(\mathcal U)$. Subsequently, 
\begin{align*}
\BFx_{\tau-1}^r(\mathcal I) &= \BFx_\tau^r(\mathcal I) - \frac{\BFx_\tau^{c}(\mathcal I)-\BFx_\tau^r(\mathcal I)+\vct\xi_\tau^r(\mathcal I)}{\tau-1}
= \BFx_\tau^r(\mathcal I) - \frac{\vct\xi_\tau^r(\mathcal I)}{\tau-1}
=\BFx_T(\mathcal I) - \bar{\vct\xi}_{\to\tau-1}^r(\mathcal I);\\
\BFx_{\tau-1}^r(\mathcal U) &= \BFx_\tau^r(\mathcal U) - \frac{\BFx_\tau^{c}(\mathcal U)-\BFx_\tau^r(\mathcal U)+\vct\xi_\tau^r(\mathcal U)}{\tau-1}
\geq \BFx_\tau^r(\mathcal U) - \frac{\vct\xi_\tau^r(\mathcal U)}{\tau-1} \geq \BFx_T(\mathcal U) - \bar{\vct\xi}_{\to\tau-1}^r(\mathcal U).
\end{align*}

Next, we prove the upper bound on $\mathbb E[T^\sharp]$. By Assumption (B4), $\{\|\bar\BFxi^r_{\to\tau}\|_{}^2\}$ is a submartingale and $\|\BFxi^r_\tau\|_{} \leq B_\xi\ a.s$. Using the identical proof by Doob's martingale inequality  in Lemma~\ref{lem:T-sharp-ub}, we have
$\mathbb E[T^\sharp] = O(1)$.
$\square$
\end{proof}

\subsection{Complete Proof of Theorem \ref{thm:upper-bound-multi-product}}
Given the initial inventory level $\BFx_T \in \mathcal{S}_\mathcal{I}$,
we first upper bound the value function of the optimal DP policy by considering a relaxed problem where the inventory of the products in $\mathcal{I}$ is $\BFx(\mathcal{I})$ but the inventory of the products in $\mathcal{U}$ is unbounded.
Let $\{\BFx_\tau^*\}_{\tau=1}^T$ be the path of inventory level process for this relaxed problem.
Define $\BFz_\tau^* := \BFx_\tau^*(\mathcal I)$.
By Eq.~\eqref{eq:bellman-star-multi-product}, the value function of this relaxed problem is given by
\[
    \Phi^*_\tau (\BFz^*_\tau) = R(\BFz^*_\tau + \BFDelta_{\tau}(\mathcal{I}))
    + \E\left[\Phi^*_{\tau-1}\left(\BFz^*_{\tau-1} -  \frac{\BFDelta_\tau(\mathcal{I})+\BFxi^*_\tau(\mathcal{I})}{\tau - 1}\right)\right]\qquad \forall \tau=T,\cdots,1.
\]
Clearly, we have $\phi^*_T(\BFx_T) \leq \Phi^*_T(\BFz_T)$. 

Below, we slightly abuse the notation and denote $\vct\Delta_\tau (\mathcal{I}),\vct\xi_\tau^*(\mathcal{I}),\vct\xi_\tau^r(\mathcal{I})$ simply by $\vct\Delta_\tau, \vct\xi_\tau^*, \vct\xi_\tau^r$.
Then, it holds that $\BFz_\tau^* = \BFz_T - \bar{\vct\Delta}_{\to\tau} - \bar{\vct\xi}_{\to\tau}^*$, where $\bar{\vct\Delta}_{\to\tau} = \frac{\vct\Delta_T}{T-1}+\cdots+\frac{\vct\Delta_{\tau+1}}{\tau}$ and $\bar{\vct\xi}_{\to\tau}^* = \frac{\vct\xi_T^*}{T-1}+\cdots +\frac{\vct\xi_{\tau+1}^*}{\tau}$.
By adapting Eq.~\eqref{eq:exp-1} to the multi-product setting, we have
\begin{equation}
    \phi_T^*(\BFx_T) \leq 
    \Phi^*_T(\BFz_T) \leq \textstyle  \E\left[\sum_{\tau=T^\sharp}^T R(\BFz_\tau^*+\vct\Delta_\tau) + (T^\sharp-1)R(\BFz_{T^\sharp-1}')\right], \label{eq:exp-1-multi}
\end{equation}
where 
\[
\BFz_{T^\sharp-1}' :=\arg\max_{\BFz \in \mathrm{proj}_{\mathcal{I}}(\mathcal{D})}\{R(\BFz) \mid \BFz\leq \BFz_{T^\sharp-1}^*\}.
\]

Next, we analyze the inventory process under the re-solving heuristic (for the original problem). Let $\BFz_\tau^r:=\BFx_\tau^r(\mathcal I)$.
By Lemma~\ref{lem:T-sharp-multi-product}, it holds that
$\BFz_\tau^r = \BFz_T - \bar{\vct\xi}_{\to\tau}^r$ for all $\tau\geq T^\sharp-1$. In addition, for all $\tau\geq T^\sharp$, it holds that $r(\BFx_\tau^c) = R(\BFz_\tau^r)$ by Lemma \ref{lem:properties-R} (recall that $\BFx_\tau^c$ is the solution of the fluid model given the right-hand side $\BFx_\tau^r$). 
We adapt Eq.~(\ref{eq:exp-2}) to the multi-product setting and get
\begin{equation}
 \textstyle \phi_T^r(\BFx_T) 
\geq \E\left[\sum_{\tau=T^\sharp}^T r(\BFx_\tau^c) \right]
= \E\left[\sum_{\tau=T^\sharp}^T R(\BFz_\tau^r) \right].
\label{eq:exp-2-multi}
\end{equation}



Generalizing the arguments from Eqs.~(\ref{eq:proof-main-lb-1},\ref{eq:proof-main-lb-2}) and using the second property of Lemma~\ref{lem:properties-R},
we obtain
\begin{align}
    & \quad 
    \textstyle\mathbb E\left[\sum_{\tau=T^\sharp}^T \left(R(\BFz_\tau^*+\vct\Delta_\tau)-R(\BFz_\tau^r)\right)\right]\nonumber\\
    &\leq\E\bigg[ \sum_{\tau=T^\sharp}^T
\bigg(\big\langle\nabla R(\BFz_T), \vct\Delta_\tau-\bar{\vct\Delta}_{\to\tau}+\bar{\vct\xi}_{\to\tau}^\delta-\vct\xi_{\tau}^\delta\big\rangle - (\bar{\vct\xi}_{\to\tau}^r)^\top\nabla^2R(\BFz_T)[\vct\Delta_\tau-\bar{\vct\Delta}_{\to\tau}]\nonumber \\
&+(\bar{\vct\xi}_{\to\tau}^r)^\top\nabla^2 R(\vct z_T) [\vct\xi_{\tau}^\delta-\bar{\vct\xi}_{\to\tau}^\delta]
+ \frac{M}{2}\|\bar{\vct\xi}_{\to\tau}^r\|_{}^2\|\vct\Delta_\tau-\bar{\vct\Delta}_{\to\tau}+\bar{\vct\xi}_{\to\tau}^\delta\|_{} - \frac{m}{2}\|\vct\Delta_\tau-\bar{\vct\Delta}_{\to\tau}+\bar{\vct\xi}_{\to\tau}^\delta\|
_2^2
\bigg)\bigg],
\label{eq:proof-multi-1}
\end{align}
where $\vct\xi_\tau^\delta := \vct\xi_\tau^r-\vct\xi_\tau^*$ and $\bar{\vct\xi}_{\to\tau}^\delta :=\frac{\vct\xi_T^\delta}{T-1}+\cdots+\frac{\vct\xi_{\tau+1}^\delta}{\tau}$.

Define $\vct\eta^{*} = -(\BFz^*_{T^\sharp -1}-\BFz_{T^\sharp -1}')^+$, $\vct\eta^{r} = ([\nabla^2 R(\BFz_T)]^{-1}\nabla R(\BFz_T) - \bar
{\vct\xi}^{r}_{\to T^\sharp -1})^+$ and
$\vct\eta = \vct\eta^{*} + \vct\eta^{r}$, where $(\BFz)^+$ denotes the element-wise positive part of $\BFz$.
Eq.~(\ref{eq:proof-at-T-sharp}) can be generalized to
\begin{align}
&\ 
R(\BFz_{T^\sharp-1}')-R(\BFz_{T^\sharp-1}^r- \eta^r)
=
R(\BFz_{T^\sharp-1}^*+\eta^*)-R(\BFz_{T^\sharp-1}^r- \eta^r)\nonumber\\
\leq &\  \langle \nabla R(\BFz_T), \vct\eta-\bar{\vct\Delta}_{\to T^\sharp -1}+\bar{\vct\xi}_{\to T^\sharp -1}^\delta\rangle - (\bar{\vct\xi}_{\to T^\sharp -1}^r+\vct\eta^r)^\top \nabla^2 R(\BFz_T)[\vct\eta-\bar{\vct\Delta}_{\to T^\sharp -1}+\bar{\vct\xi}^{\delta}_{\to T^\sharp -1}] + \frac{M^2}{8m}\|\bar{\vct\xi}_{\to T^\sharp-1}^r\|_{}^4.
\label{eq:proof-at-T-sharp-multi}
\end{align}

Let $\circ$ denote the element-wise product.
Subtracting Eq.~\eqref{eq:exp-2-multi} from Eq.~\eqref{eq:exp-1-multi} and then combining Eq.~(\ref{eq:proof-at-T-sharp-multi}) with Eq.~(\ref{eq:proof-multi-1}), we obtain
\begin{equation}
\phi_T^*(\BFx_T)-\phi_T^r(\BFx_T) \leq 
\mathbb E\left[\langle\nabla R(\BFz_T), \BFA\rangle - \nabla^2 R(\BFZ_T) \circ \BFB + \nabla^2 R(\BFZ_T) \circ \BFC + \BFD\right] + O(\mathbb E[(T^\sharp-1)]),
    \label{eq:proof-multi-2}
\end{equation}
where
\begin{align*}
\BFA \ &=\ \textstyle\sum_{\tau=T^{\sharp}}^T [\vct\Delta_\tau-\bar{\vct\Delta}_{\to\tau}+\bar{\vct\xi}_{\to\tau}^\delta-\vct\xi_{\tau}^\delta] - (T^\sharp-1)(\bar{\vct\Delta}_{\to T^\sharp-1} - \bar{\vct\xi}_{\to T^\sharp -1}^\delta) + (T^\sharp-1)\vct\eta, \\
\BFB \ &=\ \textstyle\sum_{\tau=T^{\sharp}}^T\bar{\vct\xi}_{\to\tau}^r[\vct\Delta_\tau-\bar{\vct\Delta}_{\to\tau}]^\top
- (T^\sharp-1) \bar{\vct\xi}_{\to T^\sharp-1}^r \bar{\vct\Delta}_{\to T^\sharp-1}^\top + (T^\sharp-1)(\bar{\vct\xi}_{\to T^\sharp-1}^r+\vct\eta^r)\vct\eta^\top, \\
\BFC\ &=\ \textstyle\sum_{\tau=T^{\sharp}}^T\bar{\vct\xi}_{\to\tau}^r[\vct\xi_\tau^\delta - \bar{\vct\xi}_{\to\tau}^\delta]^\top
-(T^\sharp-1) (\bar{\vct\xi}_{\to T^\sharp-1}^r+\vct\eta^r) [\bar{\vct\xi}_{\to T^\sharp-1}^\delta]^\top, \\
\BFD \ &=\ \textstyle\sum_{\tau=T^{\sharp}}^T \left(\frac{M}{2}\|\bar{\vct\xi}_{\to\tau}^r\|_{}^2\|\vct\Delta_\tau-\bar{\vct\Delta}_{\to\tau}+\bar{\vct\xi}_{\to\tau}^\delta\|_{} - \frac{m}{2}\|\vct\Delta_\tau-\bar{\vct\Delta}_{\to\tau}+\bar{\vct\xi}_{\to\tau}^\delta\|_{}^2\right)  
 + (T^\sharp-1) \frac{M^2}{8m}\|\bar{\vct\xi}_{\to T^\sharp-1}^r\|_{}^4.
\end{align*}
{
Recall that in the definitions of $\BFA,\BFB,\BFC,\BFD$, all vectors $\vct\Delta_\tau,\bar{\vct\Delta}_{\to\tau}, \vct\xi_\tau^\delta,\vct{\bar\xi}_{\to\tau}^\delta,\bar{\vct\xi}_{\to\tau}^r,\vct\eta,\vct\eta^r$ are restricted to $\mathcal I$. $\BFA$ is an $|\mathcal I|$-dimensional vector, $\BFB,\BFC$ are $|\mathcal I|\times|\mathcal I|$-dimensional matrices and $\BFD$ is a scalar.
Using the same calculations as in Sec.~\ref{sec:proofs}, $\BFA,\BFB,\BFC,\BFD$ can be reduced to
\begin{align}
\mathbb E[\BFA]
&= \E[(T^\sharp-1)\vct\eta],\label{eq:A-final-multi}\\
\mathbb E[\BFB] &= \mathbb E[(T^\sharp-1)(\bar{\vct\xi}_{\to T^\sharp-1}^r+\vct\eta^r)\vct\eta^\top],\label{eq:B-final-multi}\\
\mathbb E[\BFC] &= \mathbb E\left[-\sum_{t=T^\sharp}^T\frac{\vct\xi_t^r[\vct\xi_t^\delta]^\top}{t-1}\right] - \mathbb E\big[(T^\sharp-1)\vct\eta_r[\bar{\vct\xi}_{\to T^\sharp-1}^\delta]^\top\big].
\end{align}
Generalizing Eq.~(\ref{eq:C-intermediate-1}),
it holds that
\begin{align}
\left\|\mathbb E\left[-\sum_{t=T^\sharp}^T\frac{\vct\xi_t^r[\vct\xi_t^\delta]^\top}{t-1}\right]\right\|_{2} 
&\leq  B_{\xi} \E\left[\sum_{t=2}^T \vct{1}\{T^\sharp \leq t\}
\frac{ \sqrt{\E[\|\vct\xi_t^r-\vct\xi_t^*\|_{}^2\mid \mathcal{F}_{t+1}]}}{t-1}\right]\nonumber\\
&\leq B_\xi\mathbb E\left[\sum_{t=2}^T\vct 1\{T^\sharp\leq t\}\frac{L\|\BFx_t^*+\vct\Delta_t-\BFx_t^r\|_{}}{t-1}\right]\nonumber\\
&\leq B_\xi\mathbb E\left[\sum_{t=2}^T\vct 1\{T^\sharp\leq t\}\frac{L'\|\BFz_t^*+\vct\Delta_t(\mathcal{I})-\BFz_t^r\|_{}}{t-1}\right],
    \label{eq:C-intermediate-multi}
\end{align}
where the second inequality uses Assumption (B5) and the third inequality uses the third property in Lemma \ref{lem:properties-R} with
$L'= L_z L$.
As a result, $\|\mathbb E[\BFC]\|_{2}$ can be upper bounded as
\begin{equation}
\|\mathbb E[\BFC]\|_{2} \leq L'B_\xi\mathbb E\left[\sum_{\tau=T^\sharp}^T \frac{\|\vct\Delta-\bar{\vct\Delta}_{\to\tau}+\bar{\vct\xi}_{\to\tau}^\delta\|_{}}{\tau-1}\right]
+ \mathbb E\left[(T^\sharp-1)\|\vct\eta^r\|_{}\|\bar{\vct\xi}_{\to T^\sharp-1}^\delta\|_{}\right].
    \label{eq:C-final-multi}
\end{equation}
Because
$\nabla R(\BFx_T(\mathcal{I})) - \nabla^2 R(\BFx_T(\mathcal{I})) \bar\BFxi_{\to{T^\sharp}}^r \geq 0$ by the definition of $T^\sharp$, we have
$\|\vct\eta^{r}\|\leq B_\xi / (T^\sharp -1)\ a.s.$
Combining Eq.~(\ref{eq:proof-multi-2}) with Eqs.~(\ref{eq:A-final-multi},\ref{eq:B-final-multi},\ref{eq:C-final-multi}) and using the same derivation that leads to Eq.~(\ref{eq:proof-main-ub-5}), we have
\begin{align}
& \mathbb E\left[\langle\nabla R(\BFz_T), \BFA\rangle - \nabla^2 R(\BFZ_T) \circ \BFB + \nabla^2 R(\BFZ_T)\circ \BFC + \BFD\right]
\ \leq\  
\E\bigg[ B_\xi\|\nabla^2 R(\BFz_T)\|_{2} \|\bar{\vct\xi}^{\delta}_{\to T^\sharp-1}\|_{} \nonumber\\
&+ (T^\sharp-1)\frac{M^2}{8m}\|\bar{\vct\xi}_{\to T^\sharp-1}^r\|_{}^4 +  \sum_{\tau=T^{\sharp}}^T \frac{1}{m}
\left(\bigl(\|\nabla^2 R(\BFz_T)\|_{2}\frac{L'B_\xi}{\tau-1}\bigr)^2 + \frac{M^2}{4}\|\bar{\vct\xi}_{\to\tau}^r\|_{}^4\right) 
 \bigg]. \label{eq:proof-multi-3}
\end{align}
To complete the proof, we upper bound each term in Eq.~(\ref{eq:proof-multi-3}).
First it is easy to verify that
\begin{equation}
\sum_{\tau=T^\sharp}^T \left(\frac{\|\nabla^2 R(\BFz_T)\|_{2}L'B_\xi}{\tau-1}\right)^2 \leq (\|\nabla^2 R(\BFz_T)\|_{2}LB_\xi)^2 \sum_{j=1}^{T-1}\frac{1}{j^2} \leq 2(\|\nabla^2 R(\BFz_T)\|_{2}LB_\xi)^2\quad a.s.
\label{eq:proof-multi-6}
\end{equation}
We next focus on the terms involving $\|\bar{\vct\xi}_{\to\tau}^r\|_{}^4$. 
Note that $\{\|\bar\BFxi_{\to\tau-1}^r\|_{}^4\}$ is a submartingale adapted to the filtration $\{\mathcal{F}_{\tau}\}_{\tau=1}^T$. Let $S_t = \sum_{\tau=t}^{T} (t-1)(\|\bar\BFxi_{\to\tau-1}^r\|_{}^4-\|\bar\BFxi_{\to\tau}^r\|_{}^4)$, then $\{S_t\}$ is also a submartingale. Since $T^\sharp$ is stopping time, we have
\[
\E\left[(T^\sharp-1)\|\bar{\vct\xi}_{\to T^\sharp-1}^r\|_{}^4 +  \sum_{\tau=T^{\sharp}}^T   \|\bar{\vct\xi}_{\to\tau}^r\|_{}^4 \right]
= \E[S_{T^\sharp}]\leq \E[S_{2}]=
\E\left[\sum_{\tau=1}^T   \|\bar{\vct\xi}_{\to\tau}^r\|_{}^4 \right].
\]
It is easy to verify that
\begin{align*}
\E[\|\bar{\vct\xi}_{\to t}^r\|_{}^4]
&\leq \sum_{j,k>t}\frac{\E[\|\vct\xi_j^r\|_{}^2\|\vct\xi_k^r\|_{}^2]}{(j-1)^2(k-1)^2}
\leq B_\xi^4 \left(\sum_{j>t}\frac{1}{(j-1)^2}\right)^2 \leq \frac{4 B_\xi^4}{t^2}.
\end{align*}
Subsequently, 
\begin{align}
\E\left[\frac{M^2}{8m}(T^\sharp -1)\|\bar{\vct\xi}_{\to T^\sharp-1}^r\|_{}^4 +\frac{M^2}{4m} \sum_{\tau=T^\sharp}^T\|\bar{\vct\xi}_{\to\tau}^r\|_{}^4\right]
&\leq \frac{M^2}{4m}\mathbb \sum_{\tau=1}^T \frac{4B_\xi^4}{\tau^2} \leq \frac{M^2B_\xi^4}{2m}.
    \label{eq:proof-multi-4}
\end{align}
Since the Eucliean norm is convex, by Assumption (B4), $\|\bar{\vct\xi}^{\delta}_{\to\tau-1}\|_{}$ is a submartingale. The expectation of $\|\bar{\vct\xi}_{\to T^\sharp-1}^\delta\|_{}$ can be upper bounded by
\begin{equation}
    \E\big[ \|\bar{\vct\xi}^{\delta}_{\to T^\sharp-1}\|_{} \big] \leq 
    \E\big[ \|\bar{\vct\xi}^{\delta}_{1}\|_{} \big]
    \leq \sqrt{\E\big[ \|\bar{\vct\xi}^{\delta}_{1}\|_{}^2 \big]} \leq 
    2B_\xi \sqrt{\sum_{j=1}^{T-1}\frac{1}{j^2}} \leq {3B_\xi}.
    \label{eq:proof-multi-5}
\end{equation}
Finally, combining Eqs.~(\ref{eq:proof-multi-1},\ref{eq:proof-multi-2},\ref{eq:proof-multi-3},\ref{eq:proof-multi-4},\ref{eq:proof-multi-5}) and using Lemma~\ref{lem:T-sharp-multi-product}, we obtain
\begin{align*}
\phi_T^*(\BFx_T)-\phi_T^r(\BFx_T) \leq \frac{M^2B_\xi^4}{2m} + \frac{2(\|\nabla^2 R(\BFz_T)\|_{2}L'B_\xi)^2}{m} +
3\|\nabla^2 R(\BFz_T)\|_{2}B_\xi^2 + O(\E[T^\sharp-1]) = O(1),
\end{align*}
which completes the proof of Theorem~\ref{thm:upper-bound-multi-product}.
}

\section{Conclusion}

In this paper, we analyze a natural re-solving heuristic for the classic price-based revenue management problem with either a single product or multiple products. The heuristic re-solves the fluid model in each period to reset the prices.

{
We establish
two complementary theoretical results. 
First, the re-solving heuristic attains $O(1)$ regret compared against the value of the optimal policy. 
The $O(1)$ regret depends on the shape of the demand function as well as how close the initial inventory level is to certain ``boundaries.''
Going forward, an obvious question is whether the boundary condition can be removed. Our numerical experiment shows that the natural re-solving heuristic may not have $O(1)$ regret when the initial inventory is on the boundary, so the pricing algorithm needs to be modified for that case.

Second, we show that there exists an $\Omega(\ln T)$ gap between the value of the optimal policy and the value of the fluid model. For that reason, our regret analysis does not use the fluid model as a benchmark; the proof directly compares the value of the optimal policy and that of the heuristic.
An interesting future direction is to find a different benchmark that is within $O(1)$ of optimal value, which may help simplify our proof.
}



\section*{Appendix: The Hindsight-Optimum (HO) Benchmark.}


The HO benchmark has been used  to analyze re-solving algorithms
for quantity-based network revenue management \citep{reiman2008asymptotically,bumpensanti2020re}.
Since in quantity-based network revenue management the demand rates are not affected by the (adaptively chosen) prices,
the formulation in \cite{bumpensanti2020re} is not directly applicable to our setting.
Instead, we formulate an HO benchmark following the strategy in \cite{vera2019online} which also considered price-based revenue management
with a finite subset of pries.
\begin{definition}[The HO benchmark]
For any $p$ define random variable $D_T(p) := \sum_{t=1}^T d_t$ as the total realized demand with fixed price $p_t\equiv p$.
A policy $\pi$ is HO-admissible if at time $t$, the price decision $p_t$ depends only on $\{p_{t'},x_{t'},d_{t'}\}_{t'<t}$ and $\{D_T(p)\}_{p\in[\underline{p}, \overline{p}]}$.
The HO-benchmark $R^{\mathrm{HO}}(T,x_0)$ is defined as the expected revenue of the optimal HO-admissible policy $\pi$.
\label{defn:ho-benchmark}
\end{definition}

At a higher level, the HO-benchmark equips a policy with the knowledge of the total realized demand for \emph{each} hypothetical fixed price $p\in[\underline{p}, \overline{p}]$
in hindsight.
Clearly, such policies are more powerful than an ordinary admissible policy which only knows the \emph{expected} demand but not the realized demand
for a specific price $p$.

Our next proposition shows that the HO-benchmark $R^{\mathrm{HO}}(T,x_T)$ has a constant gap compared against
the $Tp^*f(p^*)$ oracle in the single-product setting. Hence, it also has an $\Omega(\ln T)$ gap from the re-solving heuristic and the optimal DP solution. 
The conclusion in Theorem \ref{thm:main-lb} then holds with $Tr(x_T)$ replaced by $R^{\mathrm{HO}}(T,x_T)$.

\begin{proposition}
For any $x_T\in(\underline d,x^{u})$, it holds that $R^{\mathrm{HO}}(T,x_T) \geq Tr(x_T) - O(1)$ where $y_0=x_T T$.
\label{prop:ho-bad}
\end{proposition}

\begin{proof}{Proof of Proposition \ref{prop:ho-bad}.}
Consider a setting where $\xi_1,\cdots,\xi_T$ are i.i.d. It is clear that knowing $\{D_T(p)\}_{p\in[\underline{p}, \overline{p}]}$ is equivalent to knowing $\bar\xi = \frac{1}{T}\sum_{t=1}^T\xi_t$,
since $D_T(p) = T(f(p)+\bar\xi)$ for all $p$.
Now consider the policy of fixed prices $p_t \equiv g(x_T+\bar\xi)$.
Since $\E[\bar\xi]=0$ and $\E[\bar\xi^2] = O(1/T)$,
the expected regret of such a policy can be bounded as
\begin{align*}
T\E_{\bar\xi}\left[(x_T+\bar\xi)f^{-1}(x_T+\bar\xi)\right] - Tx_Tf^{-1}(x_T)
&= T\E_{\bar\xi}[r(x_T+\bar\xi)-r(x_T)]
\geq T\E_{\bar\xi}\left[r'(x_T)\bar\xi - \frac{m}{2}\bar\xi^2\right]\\
&= - \frac{m}{2}T\E[\bar\xi^2] = -\frac{m}{2}\times O(1) = -O(1),
\end{align*}
which is to be demonstrated. $\square$
\end{proof}

%
%
%

\bibliography{refs}
\bibliographystyle{ormsv080}

\end{document}